\documentclass[finalsize]{zaa}
\usepackage{hyperref,url}
\begin{document}

\numberwithin{equation}{section}

\newcommand{\cA}{\mathcal{A}}
\newcommand{\cB}{\mathcal{B}}
\newcommand{\cE}{\mathcal{E}}
\newcommand{\cF}{\mathcal{F}}
\newcommand{\cG}{\mathcal{G}}
\newcommand{\cI}{\mathcal{I}}
\newcommand{\cJ}{\mathcal{J}}
\newcommand{\cK}{\mathcal{K}}
\newcommand{\cN}{\mathcal{N}}
\newcommand{\cU}{\mathcal{U}}
\newcommand{\cV}{\mathcal{V}}
\newcommand{\cZ}{\mathcal{Z}}

\newcommand{\mC}{\mathbb{C}}
\newcommand{\mN}{\mathbb{N}}
\newcommand{\mR}{\mathbb{R}}
\newcommand{\mZ}{\mathbb{Z}}

\newcommand{\fA}{\mathfrak{A}}
\newcommand{\fJ}{\mathfrak{J}}

\newcommand{\alg}{\operatorname{alg}}
\newcommand{\clos}{\operatorname{clos}}
\newcommand{\eps}{\varepsilon}

\title{Necessary Conditions for Fredholmness\\
of Singular Integral Operators with Shifts\\
and Slowly Oscillating Data}

\runtitle{Singular integral operators with shifts}

\author{Alexei Yu. Karlovich, Yuri I. Karlovich, and Amarino B. Lebre}
\runauthor{A. Karlovich et al.}

\address{A.~Yu.~Karlovich:
Departamento de Matem\'atica,
Faculdade de Ci\^encias e Tecnologia,
Universidade Nova de Lisboa,
Quinta da Torre,
2829--516 Caparica,
Portugal.}

\address{Yu.~I.~Karlovich:
Facultad de Ciencias,
Universidad Aut\'onoma del Estado de Morelos,
Av. Universidad 1001, Col. Chamilpa,
C.P. 62209 Cuernavaca, Morelos,
M\'exico.}

\address{A.~B.~Lebre:
Departamento de Matem\'atica,
Instituto Superior T\'ecnico,
Universidade T\'ecnica de Lisboa,
Av. Rovisco Pais,
1049--001 Lisboa,
Portugal.}

\dedication{Dedicated to Vladimir S. Rabinovich on the occasion of his 70th birthday}

\abstract{%
Suppose $\alpha$ is an orientation-preserving diffeomorphism
(shift) of $\mR_+=(0,\infty)$  onto itself with the only fixed
points $0$ and $\infty$. In \cite{KKLsufficiency} we found
sufficient conditions for the Fredholmness of the singular
integral operator with shift
\[
(aI-bW_\alpha)P_++(cI-dW_\alpha)P_-
\]
acting on $L^p(\mR_+)$ with $1<p<\infty$, where $P_\pm=(I\pm S)/2$,
$S$ is the Cauchy singular integral
operator, and $W_\alpha f=f\circ\alpha$ is the shift operator,
under the assumptions that the coefficients $a,b,c,d$ and the
derivative $\alpha'$ of the shift are bounded and continuous on
$\mR_+$ and may admit discontinuities of slowly oscillating type
at $0$ and $\infty$. Now we prove that those conditions are also necessary.
}


\primclass{45E05}
\secclasses{47A53, 47B35, 47G10, 47G30.}

\received{October 26, 2010}
\logo{00}{2010}{0}{1}{26}
\doi{3445}

\maketitle

\section{Introduction}
A bounded linear operator on a Banach space is said to be Fredholm if its
image is closed and the dimensions of its kernel and the kernel of its
adjoint operator are finite. Let $\mR_+:=(0,+\infty)$. A bounded and
continuous function $f$ on $\mR$ is called slowly oscillating (at $0$ and $\infty$)
if for each (equivalently, for some) $\lambda\in(0,1)$,
\[
\lim_{r\to s}\sup_{t,\tau\in[\lambda r,r]}|f(t)-f(\tau)|=0
\quad (s\in\{0,\infty\}).
\]
The set $SO(\mR_+)$ of all slowly oscillating functions forms a $C^*$-algebra.
This algebra properly contains $C(\overline{\mR}_+)$, the $C^*$-algebra
of all continuous functions on $\overline{\mR}_+:=[0,+\infty]$.
Suppose $\alpha$ is an orientation-preserving diffeomorphism of $\mR_+$
onto itself, which has only two fixed points $0$ and $\infty$.
We say that $\alpha$ is a slowly oscillating shift if $\log\alpha'$
is bounded and $\alpha'\in SO(\mR_+)$. The set of all slowly
oscillating shifts is denoted by $SOS(\mR_+)$.

Through the paper we suppose that $1<p<\infty$ and $1/p+1/q=1$. It is easy to see that if $\alpha\in SOS(\mR_+)$,
then the shift operator $W_\alpha$ defined by $W_\alpha f=f\circ\alpha$
is bounded and invertible on all spaces $L^p(\mR_+)$ and its inverse
is given by $W_\alpha^{-1}=W_\beta$, where $\beta:=\alpha_{-1}$ is the
inverse function to $\alpha$. It is well known that the Cauchy singular integral operator $S$ given by
\[
(Sf)(t):=\lim_{\varepsilon\to 0} \frac{1}{\pi i}
\int_{\mR_+\setminus(t-\varepsilon,t+\varepsilon)}
\frac{f(\tau)}{\tau-t}\:d\tau\quad(t\in\mR_+)
\]
is bounded on all Lebesgue spaces $L^p(\mR_+)$ for $1<p<\infty$. Put $P_\pm:=(I\pm S)/2$.

By $M(\mathfrak{A})$ denote the maximal ideal space of a unital commutative
Banach algebra $\mathfrak{A}$. Identifying the points
$t\in\overline{\mR}_+$ with the evaluation functionals $t(f)=f(t)$
for $f\in C(\overline{\mR}_+)$, we get
$M(C(\overline{\mR}_+))=\overline{\mR}_+$. Consider the fibers
\[
M_s(SO(\mR_+)):=\big\{\xi\in M(SO(\mR_+)):\xi|_{C(\overline{\mR}_+)}=s\big\}
\]
of the maximal ideal space $M(SO(\mR_+))$ over the points
$s\in\{0,\infty\}$. By \cite[Proposition~2.1]{K08}, the set
\[
\Delta:=M_0(SO(\mR_+))\cup M_\infty(SO(\mR_+))
\]
coincides with $\operatorname{clos}_{SO^*}\mR_+\setminus\mR_+$
where $\operatorname{clos}_{SO^*}\mR_+$ is the weak-star closure
of $\mR_+$ in the dual space of $SO(\mR_+)$. Then $M(SO(\mR_+))=\Delta\cup\mR_+$.
In what follows we write $a(\xi):=\xi(a)$ for every $a\in SO(\mR_+)$ and
every $\xi\in\Delta$. This paper is a continuation of our work
\cite{KKLsufficiency}, where the following result was proved.
\begin{theorem}[{\cite[Theorem~1.2]{KKLsufficiency}}]\label{th:sufficiency}
Let $a,b,c,d\in SO(\mR_+)$ and $\alpha\in SOS(\mR_+)$. The
singular integral operator
\begin{equation}\label{eq:def-N}
N:= (aI-bW_\alpha)P_++(cI-dW_\alpha)P_-
\end{equation}
with the shift $\alpha$ is Fredholm on the space $L^p(\mR_+)$ if
the following two conditions are fulfilled:
\begin{enumerate}
\item[{\rm(i)}] the functional operators $A_+:=a I-bW_\alpha$ and
$A_-:=cI-dW_\alpha$ are invertible on the space $L^p(\mR_+)$;

\item[{\rm(ii)}] for every pair $(\xi,x)\in\Delta\times\mR$,
\begin{align}
n_\xi(x)
&:=
\Big[a(\xi)-b(\xi)e^{i\omega(\xi)(x+i/p)}\Big]
\frac{1+\coth[\pi(x+i/p)]}{2}
\nonumber
\\
&\quad+
\Big[c(\xi)-d(\xi)e^{i\omega(\xi)(x+i/p)}\Big]
\frac{1-\coth[\pi(x+i/p)]}{2}\ne 0,
\label{eq:def-n}
\end{align}
where $\omega(t):=\log[\alpha(t)/t]\in SO(\mR_+)$.
\end{enumerate}
\end{theorem}
It turns out that the sufficient conditions for the Fredholmness
of the operator $N$ contained in Theorem~\ref{th:sufficiency} are also
necessary.
\begin{theorem}[Main result]
\label{th:main}
Suppose $a,b,c,d\in SO(\mR_+)$ and $\alpha\in SOS(\mR_+)$. If the operator
$N$ given by \eqref{eq:def-N} is Fredholm on $L^p(\mR_+)$, then conditions
{\rm (i)} and {\rm(ii)} of Theorem~{\rm\ref{th:sufficiency}} are fulfilled.
\end{theorem}
The proof of Theorem~\ref{th:main} is based on the method of
limit operators, which was essentially developed by V. S. Rabinovich
(see, e.g., \cite{BKR00,RRS04,L06} and the references therein), and
on the Allan-Douglas localization (see \cite{BS06}).
The paper is organized as follows. In Section~\ref{sec:SO-SOS} we collect
properties of slowly oscillating functions and slowly oscillating shifts.
In Section~\ref{sec:Mellin-convolution} we recall properties of Mellin
convolution operators with piecewise continuous and semi-almost periodic
symbols. In Section~\ref{sec:LO} we recall that if an $A$ operator is invertible
modulo some ideal $\fJ$ and the limit operators for all operators in this ideal vanish, then
the limit operator of $A$ is invertible whenever it exists. Further we calculate
the limit operators of the operator $N$ with respect to two different systems
of pseudoisometries (dilations and modulations). Let $\cK$ be the ideal of
all compact operators on $L^p(\mR_+)$. In \cite{KKLsufficiency} we introduced
the algebra $\cZ$ generated by the ideal $\cK$, the operators $I,S$, and
$cR$, where $c\in SO(\mR_+)$ and $R$ is the operator
with fixed singularities at $0$ and $\infty$ given by
\[
(R f)(t):=\frac{1}{\pi i}\int_0^\infty\frac{f(\tau)}{\tau+t}\quad (t\in\mR_+).
\]
It turns out that the algebra $\Lambda$ of all operators commuting with the
elements of $\cZ$ modulo the ideal $\cK$ contains the operator $N$.
In Section~\ref{sec:localization} we state a consequence of the Allan-Douglas
local principle  for $A\in\Lambda$, which was obtained in \cite{KKLsufficiency}.
In Section~\ref{sec:FO} we formulate an invertibility criterion for $aI-bW_\alpha$
with slowly oscillating data (Theorem~\ref{th:FO}) and
prove two auxiliary statements: a corollary of Theorem~\ref{th:FO}
related to the existence of infinite dimensional kernel or cokernel for $aI-bW_\alpha$
and a criterion for the invertibility of $aI-bW_\alpha$ with multiplicative shift
$\alpha$.  In the proof of the latter result we use limit operators with
respect to a specially chosen system of modulations, so that the limit
operator of $W_\alpha$ is equal to $W_\alpha$.
Section~\ref{sec:necessity} is devoted to the proof of Theorem~\ref{th:main}.
First we observe that the limit operators with respect to  dilations
are
\begin{equation}\label{eq:LO-introduction}
\big(a(\xi)I-b(\xi)W_{\alpha_\xi}\big)P_++
\big(c(\xi)I-d(\xi)W_{\alpha_\xi}\big)P_-,
\end{equation}
where $\xi\in\Delta$ and $\alpha_\xi(t)=e^{\omega(\xi)}t$ is a multiplicative shift.
Since $N$ is Fredholm, all limit operators are invertible for $\xi\in\Delta$.
Applying the results of Sections~\ref{sec:localization} and \ref{sec:FO},
we prove that then the operators $a(\xi)I-b(\xi)W_{\alpha_\xi}$ and
$c(\xi)I-d(\xi)W_{\alpha_\xi}$ are invertible for all $\xi\in\Delta$.
Since the fibers $M_0(SO(\mR_+))$ and $M_\infty(SO(\mR_+))$ are connected,
from the above observation and Theorem~\ref{th:FO} it follows that the operators
$aI-bW_\alpha$ and $cI-dW_\alpha$ are invertible, and this is condition (i) of
Theorem~\ref{th:sufficiency}. On the other hand, the  (invertible) limit
operators \eqref{eq:LO-introduction} are similar to the Mellin convolution
operators with the semi-almost periodic symbols $n_\xi$. Applying the invertibility
criterion for such operators (Theorem~\ref{th:invertibility-convolution}),
we arrive at condition (ii) of Theorem~\ref{th:sufficiency}.
\section{Slowly oscillating functions and shifts}\label{sec:SO-SOS}
\subsection{Properties of slowly oscillating functions}
The following two lemmas give important properties of the fibers
$M_0(SO(\mR_+))$ and $M_\infty(SO(\mR_+))$.
\begin{lemma}[{\cite[Proposition~2.2]{K08}}]
\label{le:SO-fundamental-property}
Let $\{a_k\}_{k=1}^\infty$ be a countable subset of the space $SO(\mR_+)$ and
$s\in\{0,\infty\}$. For each $\xi\in M_s(SO(\mR_+))$ there exists a
sequence $\{t_n\}\subset\mR_+$ such that $t_n\to s$ as $n\to\infty$ and
\begin{equation}\label{eq:SO-fundamental-property}
\xi(a_k)=\lim_{n\to\infty}a_k(t_n)\quad\mbox{for all}\quad k\in\mN.
\end{equation}
Conversely, if $\{t_n\}\subset\mR_+$ is a sequence such that $t_n\to s$
as $n\to\infty$, then there exists a functional $\xi\in M_s(SO(\mR_+))$
such that \eqref{eq:SO-fundamental-property} holds.
\end{lemma}
\begin{lemma}\label{le:connected-fibers}
The fibers $M_0(SO(\mR_+))$ and $M_\infty(SO(\mR_+))$ are connected
compact Hausdorff spaces.
\end{lemma}
\begin{proof}
Fix $s\in\{0,\infty\}$. Since $M_s(SO(\mR_+))$ is a closed subset of the
compact Hausdorff space $M(SO(\mR_+))$, we conclude that $M_s(SO(\mR_+))$ also
is a compact Hausdorff space. Suppose the fiber $M_s(SO(\mR_+))$ is disconnected.
Then there exist two disjoint closed subsets $X_1$ and $X_2$ such that
$M_s(SO(\mR_+))=X_1\cup X_2$. Take a continuous function $\widehat{a}$ on
$M_s(SO(\mR_+))$ such that $\widehat{a}(X_1)\subset[0,1/3]$ and
$\widehat{a}(X_2)\subset[2/3,1]$. By the Tietze extension theorem
(see e.g. \cite[Theorem~IV.11]{RS80}),
the function $\widehat{a}$ is extended to a continuous function on the whole
compact space $M(SO(\mR_+))$. We denote this extension again by $\widehat{a}$.
Because $SO(\mR_+)$ is a $C^*$-algebra, the function $\widehat{a}\in C(M(SO(\mR_+)))$
is the Gelfand transform of a function $a\in SO(\mR_+)$. Then in view of
Lemma~\ref{le:SO-fundamental-property} there are sequences $t_n',t_n''\to s$
such that there exist $\lim\limits_{n\to\infty}a(t_n')\in[0,1/3]$ and
$\lim\limits_{n\to\infty}a(t_n'')\in[2/3,1]$. Since $a\in SO(\mR_+)$
is continuous on $\mR_+$, there are points $t_n$ between $t_n'$ and $t_n''$
such that $a(t_n)=1/2$. Then $t_n\to s$, $\displaystyle\lim_{n\to\infty}
a(t_n)=1/2$, and hence $1/2\in\widehat{a}(X_1)\cup\widehat{a}(X_2)$, a
contradiction. Thus, $M_s(SO(\mR_+))$ is a connected set.
\QED
\end{proof}
Repeating literally the proofs of \cite[Proposition~3.3]{KKL03} and
\cite[Lemma~3.5]{KKL03}, we obtain the following two statements.
\begin{lemma}\label{le:SO-nec}
Suppose $\varphi\in C^1(\mR_+)$ and put $\psi(t):=t\varphi'(t)$
for $t\in\mR_+$. If $\varphi,\,\psi\in SO(\mR_+)$, then
$\lim\limits_{t\to s}\psi(t)=0$ for $s\in\{0,\infty\}$.
\end{lemma}
\begin{lemma}\label{le:SO-uniform}
Let $a\in SO(\mR_+)$. Suppose continuous functions
$f_j:\mR_+\to\mR_+$ $(j=1,2)$ and ${\cF}:\mR_+\times\mR_+\to\mR_+$
satisfy the relation
\[
xf_1(y)\le{\cF}(x,y)\le xf_2(y), \quad x,y\in\mR_+.
\]
If for some sequence $t=\{t_n\}_{n=1}^\infty$ tending to
$s\in\{0,\infty\}$ the limit
\[
\lim_{n\to\infty}a(t_n)=:a_t
\]
exists, then for every $y\in\mR_+$ the limit
$\lim\limits_{n\to\infty}a({\cF}(t_n,y))$ also exists. Moreover,
\[
\lim_{n\to\infty}a({\cF}(t_n,y))=a_t,
\]
and the convergence is uniform on every segment $J\subset\mR_+$.
\end{lemma}
\subsection{Properties of slowly oscillating shifts}
In this subsection we list necessary properties of slowly oscillating shifts.
\begin{lemma}[{\cite[Lemma~2.2]{KKLsufficiency}}]
\label{le:exp-repr}
An orientation-preserving non-Carleman shift
$\alpha:\mR_+\to\mR_+$ belongs to $SOS(\mR_+)$ if and only if
\begin{equation}\label{eq:exp-repr-1}
\alpha(t)=te^{\omega (t)},\quad t\in \mR_+,
\end{equation}
for some real-valued function $\omega\in SO(\mR_+)\cap C^1(\mR_+)$ such that
the function $t\mapsto t\omega^\prime(t)$ also belongs to $SO(\mR_+)$ and
$\inf\limits_{t\in\mR_+}\big(1+t\omega'(t)\big)>0$.
\end{lemma}
The function $\omega$ in \eqref{eq:exp-repr-1} is referred to as the exponent
function of $\alpha$.
\begin{lemma}[{\cite[Lemma~2.3]{KKLsufficiency}}]
\label{le:continuous-SOS}
Suppose $c\in SO(\mR_+)$ and $\alpha\in SOS(\mR_+)$. Then $c\circ\alpha\in SO(\mR_+)$.
\end{lemma}
\begin{lemma}[{\cite[Lemma~2.4]{KKLsufficiency}}]
\label{le:SOS-inverse}
If $\alpha\in SOS(\mR_+)$, then $\beta\in SOS(\mR_+)$.
\end{lemma}
\begin{lemma}\label{le:SOS-derivative}
Suppose $\alpha\in SOS(\mR_+)$ and $\omega$ is the exponent function of $\alpha$.
Then $\alpha'(\xi)=e^{\omega(\xi)}$ for all $\xi\in\Delta$.
\end{lemma}
\begin{proof}
By Lemma~\ref{le:exp-repr}, $\alpha(\tau)=\tau e^{\omega(\tau)}$ for $\tau\in\mR_+$
with $\omega\in SO(\mR_+)\cap C^1(\mR_+)$ and $\psi(\tau)=\tau\omega'(\tau)\in SO(\mR_+)$.
From Lemma~\ref{le:SO-nec} it follows that
\begin{equation}\label{eq:SOS-derivative-1}
\lim_{\tau\to s}\tau\omega'(\tau)
=0
\quad\mbox{for}\quad s\in\{0,\infty\}.
\end{equation}
Fix $s\in\{0,\infty\}$. By Lemma~\ref{le:SO-fundamental-property}, for a given
$\xi\in M_s(SO(\mR_+))$ there exists a sequence $t_n\to s$ as $n\to\infty$
such that
\begin{equation}\label{eq:SOS-derivative-2}
\omega(\xi) =\lim_{n\to\infty}\omega(t_n),
\quad
\alpha'(\xi)=\lim_{n\to\infty}\alpha'(t_n)
\end{equation}
(recall that $\alpha'\in SO(\mR_+)$).
Clearly, $\alpha'(\tau)=(1+t\omega'(\tau))e^{\omega(\tau)}$ for $\tau\in\mR_+$.
Combining this relation with \eqref{eq:SOS-derivative-1}--\eqref{eq:SOS-derivative-2},
we get $\alpha'(\xi)=e^{\omega(\xi)}$.
\QED
\end{proof}
\section{Convolution operators}\label{sec:Mellin-convolution}
\subsection{Fourier convolution operators}
For a Banach space $X$, let $\cB(X)$ be the Banach algebra of all bounded linear
operators on $X$ and let $\cK(X)$ be the closed two-sided ideal of all
compact operators on $X$.

Let $F:L^2(\mR)\to L^2(\mR)$ denote the Fourier transform,
\[
(Ff)(x):=\int_\mR f(y)e^{-ixy}dy\quad (x\in\mR),
\]
and let $F^{-1}:L^2(\mR)\to L^2(\mR)$ be the inverse of $F$. A function
$a\in L^\infty(\mR)$ is called a Fourier multiplier if the map
$f\mapsto F^{-1}aFf$ maps $L^2(\mR)\cap L^p(\mR)$ onto itself and extends
to a bounded operator on $L^p(\mR)$. The latter operator is then denoted by
$W^0(a)$. We let $M_p(\mR)$ stand for the set of all Fourier multipliers on
$L^p(\mR)$. One can show that $M_p(\mR)$ is a Banach algebra under the norm
\[
\|a\|_{M_p(\mR)}:=\|W^0(a)\|_{\cB(L^p(\mR))}.
\]
\subsection{Mellin convolution operators}
Let $d\mu(t)=dt/t$ be the (normalized) invariant measure on $\mR_+$.
Consider the Fourier transform on $L^2(\mathbb{R}_+,d\mu)$, which is
usually referred to as the Mellin transform and is defined by
\[
M:L^2(\mR_+,d\mu)\to L^2(\mR),
\quad
(Mf)(x)=\int_{\mR_+} f(t) t^{-ix}\,\frac{dt}{t}.
\]
It is an invertible operator, with inverse given by
\[
{M^{-1}}:L^2(\mR)\to L^2(\mR_{+},d\mu),
\quad
({M^{-1}}g)(t)= \frac{1}{2\pi}\int_{\mR}
g(x)t^{ix}\,dx.
\]
Let $E$ be the isometric isomorphism
\begin{equation}\label{eq:def-E}
E:L^p(\mR_+,d\mu)\to L^p(\mR),
\quad
(Ef)(x):=f(e^x)\quad (x\in\mR).
\end{equation}
Then the map $A\mapsto E^{-1}AE$ transforms the Fourier convolution
operator given by $W^0(a)=F^{-1}aF$ to the Mellin convolution operator
\[
\operatorname{Co}(a):=M^{-1}aM
\]
with the same symbol $a$. Hence the class of Fourier multipliers on
$L^p(\mR)$ coincides with the class of Mellin multipliers on $L^p(\mR_+,d\mu)$.
\subsection{Piecewise continuous multipliers}
We denote by $PC$ the $C^*$-algebra of all bounded piecewise continuous
functions on $\dot{\mR}ª=\mR\cup\{\infty\}$. By definition, $a\in PC$ if
and only if $a\in L^\infty(\mR)$ and the one-sided limits
\[
a(x_0-0):=\lim_{x\to x_0-0}a(x),
\quad
a(x_0+0):=\lim_{x\to x_0+0}a(x)
\]
exist for each $x_0\in\dot{\mR}$. If a function $a$ is given everywhere on $\mR$, then
its total variation of $a$ is defined by
\[
V(a):=\sup\sum_{k=1}^n|a(x_k)-a(x_{k-1})|,
\]
where the supremum is taken over all $n\in\mN$ and
\[
-\infty<x_0<x_1<\dots<x_n<+\infty.
\]
If $a$ has a finite total variation, then it has finite one-sided limits
$a(x-0)$ and $a(x+0)$ for all $x\in\dot{\mR}$, that is, $a\in PC$.
If $a$ is an absolutely continuous function of finite total variation on $\mR$,
then $a'\in L^1(\mR)$ and
\[
V(a)=\int_\mR|a'(x)|dx
\]
(see, e.g., \cite[Chap. VIII, Sections 3 and 9; Chap. XI, Section~4]{N55}).

The following theorem gives an important subset of $M_p(\mR)$.
Its proof can be found, e.g., in \cite[Theorem~17.1]{BKS02}.
\begin{theorem}[Stechkin's inequality]
\label{th:Stechkin}
If $a\in PC$ has finite total variation $V(a)$, then $a\in M_p(\mR)$ and
\[
\|a\|_{M_p(\mR)}\le\|S_\mR\|_{\cB(L^p(\mR))}\big(\|a\|_{L^\infty(\mR)}+V(a)\big),
\]
where $S_\mR$ is the Cauchy singular integral operator on $\mR$.
\end{theorem}
According to \cite[p.~325]{BKS02}, let $PC_p$ be the closure in $M_p(\mR)$
of the set of all functions $a\in PC$ with finite total variation on $\mR$.
Following \cite[p.~331]{BKS02}, put
\[
C_p(\overline{\mR}):=PC_p\cap C(\mR),\quad \overline{\mR}:=[-\infty,+\infty].
\]
\subsection{Algebra generated by the Cauchy singular integral operator}
Suppose $\mathfrak{A}$ is a Banach algebra and $\mathfrak{S}$ is a subset of
$\mathfrak{A}$. Let $\alg_\mathfrak{A}\mathfrak{S}$ denote the smallest
closed subalgebra of  $\mathfrak{A}$ containing $\mathfrak{S}$ and
let $\operatorname{id}_\mathfrak{A}\mathfrak{S}$ denote the smallest closed
two-sided ideal of $\mathfrak{A}$ containing $\mathfrak{S}$.

Let $\cB:=\cB(L^p(\mR_+))$, $\cK:=\cK(L^p(\mR_+))$, and $\cA:=\alg_\cB\{I,S\}$.
Consider the isometric isomorphism
\begin{equation}\label{eq:def-Phi}
\Phi:L^p(\mR_+)\to L^p(\mR_+,d\mu),
\quad
(\Phi f)(t):=t^{1/p}f(t)\quad(t\in\mR_+).
\end{equation}
The following fact is well known (see, e.g., \cite[Section~2]{RS90}).
\begin{theorem}\label{th:algebra-A}
The algebra $\cA$ is the smallest closed subalgebra of $\cB$ that contains
the operators $\Phi^{-1}\operatorname{Co}(a)\Phi$ with $a\in C_p(\overline{\mR})$.
The functions
\[
s_p(x):=\coth[\pi(x+i/p)]
\quad
r_p(x):=1/\sinh[\pi(x+i/p)]
\quad(x\in\mR)
\]
belong to $C_p(\overline{\mR})$ and the operators $S$ and $R$ are similar to
the Mellin convolution operators:
\[
\Phi S\Phi^{-1}=\operatorname{Co}(s_p),
\quad
\Phi R\Phi^{-1}=\operatorname{Co}(r_p).
\]
\end{theorem}
From $s_p^2-r_p^2=1$ and Theorem~\ref{th:algebra-A} it follows that
\begin{equation}\label{eq:S-R-relation}
4P_+P_-=4P_-P_+=I-S^2=-R^2.
\end{equation}
\begin{theorem}[{\cite[Corollary~6.4]{KKLsufficiency}}]
\label{th:compactness-commutators}
If $a\in SO(\mR_+)$ and $\alpha\in SOS(\mR_+)$, then for every $A\in\cA$
the operators $aA-AaI$ and $W_\alpha A-AW_\alpha$ are compact.
\end{theorem}
\subsection{Semi-almost periodic multipliers}
The following simple statement motivates us to enlarge the class of
piecewise continuous multipliers.
\begin{lemma}\label{le:mult-shift-convolution}
Let $\alpha:\mR_+\to\mR_+$ be a multiplicative shift given by $\alpha(t)=kt$
for all $t\in\mR_+$ with some $k\in\mR_+$. Then
$\Phi W_\alpha \Phi^{-1}=\operatorname{Co}(m)$ with $m(x):=e^{i(x+i/p)\log k}$
for $x\in\mR$.
\end{lemma}
\begin{proof}
The proof is a matter of a direct calculation.
\QED
\end{proof}
A function $p:\mR\to\mC$ of the form
$p(x)=\sum_{\lambda\in\Omega} r_\lambda e^{i\lambda x}$,
where $r_\lambda\in\mC$, $\lambda\in\mR$, and $\Omega$ is a finite
subset of $\mR$, is called an almost periodic polynomial.
The set of all almost periodic polynomials is denoted by $AP_0$.
From Lemma~\ref{le:mult-shift-convolution} it follows that $AP_0\subset M_p(\mR)$.
According to \cite[p.~372]{BKS02}, $AP_p$ denotes the closure of the set
of all almost periodic polynomials in the norm of $M_p(\mR)$ and $SAP_p$ denotes
the smallest closed subalgebra of $M_p(\mR)$ that contains $C_p(\overline{\mR})$
and $AP_p$.

Applying the inverse closedness of the algebra $SAP_p$ in $L^\infty(\mR)$
(see \cite[Proposition~19.4]{BKS02}), we immediately get the following.
\begin{theorem}\label{th:invertibility-convolution}
Suppose $a\in SAP_p$. The Mellin convolution operator
$\operatorname{Co}(a)$ is invertible on the space $L^p(\mR_+,d\mu)$
if and only if $\inf\limits_{x\in\mR}|a(x)|>0$.
\end{theorem}
\section{Limit operators}\label{sec:LO}
\subsection{Abstract approach}
In our previous work \cite{KKL03} the techniques of limit operators
(see, e.g., \cite{BKR00,L06,RRS04}) was successfully used to
study the invertibility of binomial functional operators that are
now the coefficients of the singular integral operator with shift
$N$ given by \eqref{eq:def-N}. In what follows we make use of such
techniques to obtain a necessary condition for the Fredholmness of
the operator $N$. Let us recall the abstract version of such
techniques.

Let $X$ be a Banach space and let $X^*$ be its dual space.
We say that an operator $U\in\cB(X)$ is a \textit{pseudoisometry} if
$U$ is invertible in $\cB(X)$ and
\[
\|U\|_{\cB(X)}=1/\|U^{-1}\|_{\cB(X)}.
\]
Let $A\in\cB(X)$ and $\cU=\{U_n\}_{n=1}^\infty$ be a sequence of
pseudoisometries. If the strong limits
\begin{equation}\label{eq:LO-defi}
\begin{split}
A_\cU :=\operatornamewithlimits{s-lim}_{n\to\infty}
(U_n^{-1}AU_n)\;\mbox{ in }\; \cB(X),
\quad
A_{\cU^*} :=\operatornamewithlimits{s-lim}_{n\to\infty}
(U_n^{-1}AU_n)^* \;\mbox{ in }\;
\cB(X^*)
\end{split}
\end{equation}
exist, then always $(A_\cU)^*=A_{{\cU}^*}$, and we will refer the
operator $A_\cU$ to as a {\it limit operator} for the operator $A$
with respect to the sequence $\cU$. Note that usually the
limit operator $A_\cU$ is defined independently of the existence
of the strong limit $A_{{\cU}^*}$ (see, e.g., \cite{BKR00,RRS04}),
while we need the existence of the both limits
\eqref{eq:LO-defi} for our purposes. If the limit operator $A_\cU$
exists, then it is uniquely determined by $A$ and $\cU$, which
justifies the notation $A_\cU$.

In the next statement we collect basic properties of limit operators.
\begin{lemma}\label{le:LO-properties}
Suppose $\cU=\{U_n\}_{n=1}^\infty\subset\cB(X)$ is a sequence of pseudo\-isometries.
\begin{enumerate}
\item[{\rm(a)}]
If $A\in\cB(X)$ and $A_\cU$ exists, then $\|A_\cU\|_{\cB(X)}\le\|A\|_{\cB(X)}$.

\item[{\rm(b)}]
If $A,B\in\cB(X)$, $\alpha\in\mC$, and if the limit operators $A_\cU$, $B_\cU$
exist, then the limit operators $(\alpha A)_\cU$, $(A+B)_\cU$, $(AB)_\cU$
also exist and
\[
(\alpha A)_\cU=\alpha A_\cU, \quad
(A+B)_\cU=A_\cU+B_\cU, \quad
(AB)_\cU=A_\cU B_\cU.
\]

\item[{\rm (c)}]
If $A\in\cB(X)$ and if $\{A_m\}_{m=1}^\infty\subset\cB(X)$ is such that
the limit operators $(A_m)_\cU$ exist for all $m\in\mN$ and $\|A-A_m\|_{\cB(X)}\to 0$
as $m\to\infty$, then the limit operator $A_\cU$ exists and $\|A_\cU-(A_m)_\cU\|_{\cB(X)}\to 0$
as $m\to \infty$.
\end{enumerate}
\end{lemma}
The proofs of the above results can be found in \cite[Proposition~3.4]{L06}
or \cite[Proposition~1.2.2]{RRS04}.
\begin{theorem}\label{th:inv-quotient-algebra}
Let $X$ be a Banach space, let $\fA$ be a closed subalgebra of $\cB(X)$,
and let $\fJ$ be a closed two-sided ideal of $\fA$. Suppose $A\in\fA$
and $\cU=\{U_n\}_{n=1}^\infty\subset\cB(X)$ is a sequence of pseudoisometries such that
the limit operator $A_\cU$ exists and the limit operators $J_\cU$ exist
and are equal to zero for all $J\in\fJ$. If the coset $A+\fJ$ is invertible
in the quotient algebra $\fA/\fJ$, then the limit operator $A_\cU$ is
invertible.
\end{theorem}
The proof is developed by analogy with \cite[Proposition~1.2.9]{RRS04}.
\subsection{Strong convergence of shift operators}
To calculate limit operators for the shift operator $W_\alpha$, we
need a result on the strong convergence of shift operators.
\begin{lemma}\label{le:strong-shift}
Let $\alpha_n:\mR_+\to\mR_+$ for $n\in\mN\cup\{0\}$ be
orientation-preserving diffeomorphisms having only two fixed points $0$ and
$\infty$, and $\beta_n$ be their inverses. If
$\log\alpha_n'\in L^\infty(\mR_+)$ for all $n\in\mN\cup\{0\}$ and
\begin{itemize}
\item[{\rm (i)}]
$\displaystyle\sup_{n\in\mN\cup\{0\}}\|\beta_n'\|_{L^\infty(\mR_+)}<\infty,$

\item[{\rm (ii)}]
$\alpha_n\to\alpha_0$ pointwise on $\mR_+$ as $n\to\infty$;
\end{itemize}
then the sequence of shift operators $W_{\alpha_n}\in\cB$
converges strongly to the shift operator $W_{\alpha_0}\in\cB$.
\end{lemma}
\begin{proof}
The idea of the proof is borrowed from \cite[Theorem~1]{DS98}.
Let $\chi_E$ denote the characteristic function of a set $E\subset\mR_+$.
Since the linear span of the set $\{\chi_{[0,\tau]}:\tau\in\mR_+\}$ is
dense in the space $L^p(\mR_+)$ and the operators $W_{\alpha_n}$ are
uniformly bounded on $L^p(\mR_+)$ in view of (i), it is sufficient to prove that
\begin{equation}\label{eq:strong-shift-1}
\lim_{n\to\infty}
\big\|W_{\alpha_n}\chi_{[0,\tau]}-W_{\alpha_0}\chi_{[0,\tau]}\big\|_{L^p(\mR_+)}=0
\quad\text{for all}\quad\tau\in\mR_+.
\end{equation}
It is easy to see that
\begin{align}
\big\|W_{\alpha_n}\chi_{[0,\tau]} &-W_{\alpha_0}\chi_{[0,\tau]}\big\|_{L^p(\mR_+)}^p
=
\int_{\mR_+}\big|\chi_{[0,\tau]}(\alpha_n(t))-\chi_{[0,\tau]}(\alpha_0(t))\big|^p\,dt
\nonumber
\\
&=
\int_{\mR_+}\big|\chi_{[0,\beta_n(\tau)]}(t)-\chi_{[0,\beta_0(\tau)]}(t)\big|^p\,dt
=
\big|\beta_n(\tau)-\beta_0(\tau)\big|.
\label{eq:strong-shift-2}
\end{align}
On the other hand,
\begin{align}
\big|\beta_n(\tau)-\beta_0(\tau)\big|
&=
\Big|
\beta_n \big[\alpha_0(\beta_0(\tau))\big]-\beta_n\big[\alpha_n(\beta_0(\tau))\big]\Big|
\nonumber
\\
&\le
\sup_{n\in\mN}\big\|\beta_n'\big\|_{L^\infty(\mR_+)}
\big|\alpha_0(\beta_0(\tau))-\alpha_n(\beta_0(\tau))\big|.
\label{eq:strong-shift-3}
\end{align}
From \eqref{eq:strong-shift-3} and the hypotheses of the lemma it follows that
\[
|\beta_n(\tau)-\beta_0(\tau)|=o(1)\quad\mbox{as}\quad n\to\infty
\]
for every $\tau\in\mR_+$. Combining this with \eqref{eq:strong-shift-2}, we arrive
at \eqref{eq:strong-shift-1}.
\QED
\end{proof}
\subsection{Realization with dilations}
For $x\in\mR_+$, consider the dilation
operator $V_x$ defined on $L^p(\mR_+)$ by
\[
(V_x f)(t):=f(t/x)\quad (t\in\mR_+).
\]
It is easy to see that $V_x$ is invertible on the space $L^p(\mR_+)$ and
$V_x^{-1}=V_{1/x}$. Moreover, $\|V_x\|_{\cB}=x^{1/p}$ and hence $V_x$ is
a pseudoisometry for every $x\in\mR_+$.

Fix $s\in\{0,\infty\}$. We say that a sequence $h:=\{h_n\}_{n=1}^\infty\subset\mR_+$
is a test sequence relative to the point $s$ if
\[
\lim_{n\to\infty}h_n=s.
\]
With each test sequence $h$ relative to the point $s$ we associate
the sequence of pseudoisometries $\cV_h^s:=\{V_{h_n}\}_{n=1}^\infty\subset\cB$.
\begin{lemma}\label{le:LO-compact-dilations}
Let $h:=\{h_n\}_{n=1}^\infty\subset\mR_+$ be a test sequence relative to
$s\in\{0,\infty\}$. For any operator $K\in\cK$, the limit operator
$K_{\cV_h^s}$ with respect to the sequence of pseudoisometries
$\cV_h^s:=\{V_{h_n}\}_{n=1}^\infty\subset\cB$ exists and is the zero operator.
\end{lemma}
\begin{proof}
Consider the isometric isomorphism $E\Phi:L^p(\mR_+)\to L^p(\mR)$,
where $E$ is defined by \eqref{eq:def-E} and $\Phi$ is defined by
\eqref{eq:def-Phi}. Then $\widetilde{K}:=E\Phi K\Phi^{-1}E^{-1}$
is compact on $L^p(\mR)$ for every $K\in\cK(L^p(\mR_+))$, and for
every $x\in\mR_+$, $E\Phi V_x\Phi^{-1}E^{-1}=\widetilde{V}_x$, where
$\widetilde{V}_x\in\cB(L^p(\mR))$ and $(\widetilde{V}_xf)(y)=x^{1/p}f(y-\log x)$
for all $y\in\mR$. By \cite[Lemma~18.9]{BKS02},
$\operatornamewithlimits{s-lim}\limits_{n\to\infty}\widetilde{V}_{h_n}^{-1}
\widetilde{K}\widetilde{V}_{h_n}I=0$ on $L^p(\mR)$ for every test sequence
$h=\{h_n\}_{n=1}^\infty\subset\mR_+$. Therefore
\[
\operatornamewithlimits{s-lim}_{n\to\infty}
V_{h_n}^{-1}KV_{h_n}=
\operatornamewithlimits{s-lim}_{n\to\infty}
\Phi^{-1}E^{-1}\widetilde{V}_{h_n}^{-1}\widetilde{K}\widetilde{V}_{h_n}E\Phi=0
\quad\mbox{on}\quad L^p(\mR_+).
\]
Analogously, $(V_{h_n}^{-1}KV_{h_n})^*=V_{h_n}^{-1}K^*V_{h_n}$
converges strongly to zero on the space $L^q(\mR_+)$. Thus,
the limit operator $K_{\cV_h^s}$ exists and is equal to zero.
\QED
\end{proof}
\begin{lemma}\label{le:LO-dilations}
Suppose $a,b,c,d\in SO(\mR_+)$, $\alpha\in SOS(\mR_+)$, and the operator
$N$ is given by \eqref{eq:def-N}. Let $s\in\{0,\infty\}$. For every
functional $\xi\in M_s(SO(\mR_+))$ there exists a test sequence
$h^\xi=\{h_n^\xi\}_{n=1}^\infty\subset\mR_+$ relative to the point
$s$ such that the limit operator $N_{\cV_{h^\xi}^s}$ with respect
to the sequence of pseudoisometries
$\cV_{h^\xi}^s:=\{V_{h_n^\xi}\}_{n=1}^\infty\subset\cB$ exists and
\begin{equation}\label{eq:LO-dilations-1}
N_{\cV_{h^\xi}^s}=
\big(a(\xi)I-b(\xi)W_{\alpha_\xi}\big)P_+ +
\big(c(\xi)I-d(\xi)W_{\alpha_\xi}\big)P_-,
\end{equation}
where
$\alpha_\xi(t):=e^{\omega(\xi)}t$ and  $\omega(t):=\log[\alpha(t)/t]$
for $t\in\mR_+$.
\end{lemma}
\begin{proof}
Fix $s\in\{0,\infty\}$ and $\xi\in M_s(SO(\mR_+))$. From Lemma~\ref{le:SOS-inverse}
it follows that $\beta:=\alpha_{-1}$ is a slowly oscillating shift.
Then, by Lemma~\ref{le:exp-repr}, $\alpha',\beta'\in SO(\mR_+)$ and the
functions $\omega$ and $\zeta(t):=\log[\beta(t)/t]$
are real-valued slowly oscillating functions. Clearly,
$\overline{a},\overline{b}, \overline{c},\overline{d}\in SO(\mR_+)$. Let
\[
\cG:=\{a,b,c,d,\alpha',\beta',\omega,\zeta\}
\subset SO(\mR_+).
\]
By Lemma~\ref{le:SO-fundamental-property}, there exists a test sequence
$h^\xi=\{h_n^\xi\}_{n=1}^\infty\subset\mR_+$ relative to the point $s$
such that the limit
\begin{equation}\label{eq:LO-dilations-2}
g(\xi)=\xi(g)=\lim_{n\to\infty}g(h_n^\xi)
\end{equation}
exists for every function $g\in\cG$. Lemma~\ref{le:SO-uniform} implies
that for every $t\in\mR_+$,
\begin{equation}\label{eq:LO-dilations-3}
\lim_{n\to\infty}|g(h_n^\xi t)-g(h_n^\xi)|=0,
\end{equation}
and the convergence is uniform in $t$ on every segment $J\subset\mR_+$.
Since
\[
V_{h_n^\xi}^{-1}(gI)V_{h_n^\xi}=g_nI,
\quad
(V_{h_n^\xi}(gI)V_{h_n^\xi})^*=\overline{g_n}I,
\]
where $g_n(t):=g(h_n^\xi t)$ for all $t\in\mR_+$, we infer from
\eqref{eq:LO-dilations-2} and \eqref{eq:LO-dilations-3} that for the
multiplication operator on $L^p(\mR_+)$ and its adjoint on $L^q(\mR_+)$,
\[
\begin{split}
&
\operatornamewithlimits{s-lim}_{n\to\infty}
V_{h_n^\xi}^{-1}(gI)V_{h_n^\xi}
=
\operatornamewithlimits{s-lim}_{n\to\infty}g_nI
=
\lim_{n\to\infty}g(h_n^\xi)I=g(\xi)I,
\\
&
\operatornamewithlimits{s-lim}_{n\to\infty}
(V_{h_n^\xi}^{-1}(gI)V_{h_n^\xi})^*
=
\operatornamewithlimits{s-lim}_{n\to\infty}\overline{g_n}I
=
\lim_{n\to\infty}\overline{g(h_n^\xi)}I=\overline{g(\xi)}I.
\end{split}
\]
Hence
\begin{equation}\label{eq:LO-dilations-4}
(gI)_{\cV_{h^\xi}^s}=g(\xi)I
\quad\mbox{for}\quad g\in\{a,b,c,d\}.
\end{equation}
From Lemma~\ref{le:exp-repr} it follows that $\alpha(t)=te^{\omega(t)}$.
Therefore, for all $n\in\mN$,
\begin{equation}\label{eq:LO-dilations-5}
V_{h_n^\xi}^{-1}W_\alpha V_{h_n^\xi}=W_{\alpha_\xi^{(n)}},
\end{equation}
where $\alpha_\xi^{(n)}(t):=te^{\omega(h_n^\xi t)}$ for $t\in\mR_+$.
From \eqref{eq:LO-dilations-2}--\eqref{eq:LO-dilations-3} we conclude that
\begin{equation}\label{eq:LO-dilations-6}
\lim_{n\to\infty}\alpha_\xi^{(n)}(t)=te^{\omega(\xi)}=\alpha_\xi(t),
\quad t\in\mR_+.
\end{equation}
Since $\log\alpha'$ is bounded, we have
\begin{equation}\label{eq:SOS-derivative*}
0<m_\alpha:=\inf_{y\in\mR_+}\alpha'(y).
\end{equation}
Let $\beta_\xi^{(n)}$ be the inverse shift to $\alpha_\xi^{(n)}$. It is
easy to see that for all $n\in\mN$ and $t\in\mR_+$,
\[
(\beta_\xi^{(n)})'(t)
=
\frac{1}{(\alpha_\xi^{(n)})'[\beta_\xi^{(n)}(t)]}
=
\frac{1}{\alpha'[h_n^\xi\beta_\xi^{(n)}(t)]}.
\]
From this equality and \eqref{eq:SOS-derivative*} it follows that for all
$n\in\mN$ and $t\in\mR_+$,
\begin{equation}\label{eq:LO-dilations-7}
(\beta_\xi^{(n)})'(t)\le 1/m_\alpha<+\infty.
\end{equation}
Moreover, the derivative of the inverse shift to $\alpha_\xi$ is constant.
Thus, combining \eqref{eq:LO-dilations-5}--\eqref{eq:LO-dilations-7} and
Lemma~\ref{le:strong-shift}, we see that
\begin{equation}\label{eq:LO-dilations-8}
\operatornamewithlimits{s-lim}_{n\to\infty}
V_{h_n^\xi}^{-1}W_\alpha V_{h_n^\xi}=W_{\alpha_\xi}
\quad\mbox{on}\quad L^p(\mR_+).
\end{equation}
It is easy to see that $(W_\alpha)^*=\beta'W_\beta$. We have already proved
that
\begin{equation}\label{eq:LO-dilations-9}
\operatornamewithlimits{s-lim}_{n\to\infty}
V_{h_n^\xi}^{-1}(\beta' I)V_{h_n^\xi}=\beta'(\xi)I
\quad\mbox{on}\quad L^q(\mR_+).
\end{equation}
Analogously to \eqref{eq:LO-dilations-8} one can show that
\begin{equation}\label{eq:LO-dilations-10}
\operatornamewithlimits{s-lim}_{n\to\infty}
V_{h_n^\xi}^{-1}W_\beta V_{h_n^\xi}=W_{\beta_\xi}
\quad\mbox{on}\quad L^q(\mR_+),
\end{equation}
where $\beta_\xi(t)=te^{\zeta(\xi)}$.
From \eqref{eq:LO-dilations-9}--\eqref{eq:LO-dilations-10} we obtain
\begin{align}
\operatornamewithlimits{s-lim}_{n\to\infty}\left(V_{h_n^\xi}^{-1} W_\alpha V_{h_n^\xi}\right)^*
&=
\left(\operatornamewithlimits{s-lim}_{n\to\infty}V_{h_n^\xi}^{-1} (\beta'I) V_{h_n^\xi}\right)
\left(\operatornamewithlimits{s-lim}_{n\to\infty} V_{h_n^\xi}^{-1} W_\beta V_{h_n^\xi}\right)\nonumber
\\
&=
\beta'(\xi)W_{\beta_\xi}.
\label{eq:LO-dilations-12}
\end{align}
Equalities \eqref{eq:LO-dilations-8} and \eqref{eq:LO-dilations-12} imply
that
\begin{equation}\label{eq:LO-dilations-13}
(W_\alpha)_{\cV_{h^\xi}^s}=W_{\alpha_\xi}.
\end{equation}
It is easy to see that $S^*=S$ and $V_{h_n^\xi}^{-1}SV_{h_n^\xi}=S$. Hence
\begin{equation}\label{eq:LO-dilations-14}
(P_+)_{\cV_h^\xi}=P_+,
\quad
(P_-)_{\cV_h^\xi}=P_-.
\end{equation}
Combining \eqref{eq:LO-dilations-4}, \eqref{eq:LO-dilations-13}, and
\eqref{eq:LO-dilations-14} with Lemma~\ref{le:LO-properties}(b),
we see that the limit operator $N_{\cV_{h^\xi}^s}$ exists and is calculated by
\eqref{eq:LO-dilations-1}.
\QED
\end{proof}
\subsection{Realization with modulations}
For $x\in\mR$, consider the modulation
operator $E_x$ defined on $L^p(\mR_+)$ by
\[
(E_xf)(t):=t^{ix}f(t)\quad (t\in\mR_+).
\]
It is clear that $E_x$ is invertible on the space $L^p(\mR_+)$ and
$E_x^{-1}=E_{-x}$. Moreover, $\|E_x\|_{\cB}=1$ and hence $E_x$ is a
pseudoisometry for each $x\in\mR$.

Fix $s\in\{-\infty,+\infty\}$. We say that a sequence $\mu:=\{\mu_n\}_{n=1}^\infty\subset\mR$
is a test sequence relative to the point $s$ if
\[
\lim_{n\to\infty}\mu_n=s.
\]
With each test sequence $\mu$ relative to the point $s$ we associate
the sequence of pseudoisometries $\cE_\mu^s:=\{E_{\mu_n}\}_{n=1}^\infty\subset\cB$.
\begin{lemma}\label{le:LO-compact-modulations}
Suppose $\mu:=\{\mu_n\}_{n=1}^\infty\subset\mR$ is a test sequence
relative to a point $s\in\{-\infty,+\infty\}$. For any operator
$K\in\cK$, the limit operator $K_{\cE_\mu^s}$ with respect to the
sequence of pseudoisometries
$\cE_\mu^s:=\{E_{\mu_n}\}_{n=1}^\infty\subset\cB$ exists and is
the zero operator.
\end{lemma}
\begin{proof}
Following the proof of Lemma~\ref{le:LO-compact-dilations}, consider
the isometric isomorphism $E\Phi:L^p(\mR_+)\to L^p(\mR)$ and the operator
$\widetilde{K}:=E\Phi K\Phi^{-1}E^{-1}\in\cK(L^p(\mR))$, where
$K\in\cK(L^p(\mR_+))$. For every $x\in\mR$ we get $E\Phi E_x\Phi^{-1}E^{-1}=e_xI,$
where $e_x(y):=e^{ixy}$ for $y\in\mR$. It was shown in \cite[Lemma~3.8]{KL08}
(see also \cite[Lemma~10.1]{BKS02} for $p=2$) that
$\operatornamewithlimits{s-lim}\limits_{n\to\infty}e_{-\mu_n}\widetilde{K}e_{\mu_n}I=0$
on $L^p(\mR)$ for every sequence $\mu_n$ tending to $+\infty$ or to $-\infty$. Therefore
\[
\operatornamewithlimits{s-lim}_{n\to\infty}
E_{\mu_n}^{-1}KE_{\mu_n}=
\operatornamewithlimits{s-lim}_{n\to\infty}
\Phi^{-1}E^{-1}e_{-\mu_n}\widetilde{K}e_{\mu_n}E\Phi=0
\quad\mbox{on}\quad L^p(\mR_+).
\]
Analogously, $(E_{\mu_n}^{-1}KE_{\mu_n})^*=E_{\mu_n}^{-1}K^*E_{\mu_n}$
converges strongly to zero on the space $L^q(\mR_+)$. Thus, for every sequence
$\mu=\{\mu_n\}_{n=1}^\infty\subset\mR$ converging to $s\in\{-\infty,+\infty\}$,
we have $K_{\cE_\mu^s}=0$.
\QED
\end{proof}
\begin{lemma}\label{le:LO-modulations}
Suppose $g\in SO(\mR_+)$.
Let $\mu=\{\mu_n\}_{n=1}^\infty\subset\mR$ be a test sequence relative to $s\in\{-\infty,+\infty\}$.
Then the limit operators $(gI)_{\cE_\mu^s}$ and $S_{\cE_\mu^s}$
with respect to the sequence of pseudo\-isometries
$\cE_\mu^s:=\{E_{\mu_n}\}_{n=1}^\infty\subset\cB$ exist and are given by
\[
(gI)_{\cE_\mu^{\pm\infty}}=gI,
\quad
S_{\cE_\mu^{\pm\infty}}=\pm I.
\]
\end{lemma}
\begin{proof}
Let $s=+\infty$. It is easy to see that
\[
E_{\mu_n}^{-1}(gI)E_{\mu_n}=gI,
\quad
(E_{\mu_n}^{-1}(gI)E_{\mu_n})^*=(gI)^*
\]
for all $n\in\mN$. Hence
$(gI)_{\cE_\mu^{+\infty}}=gI$.

Let us show that $S_{\cE_\mu^{+\infty}}=I$. From Theorem~\ref{th:algebra-A}
it follows that
\[
E_{\mu_n}^{-1}SE_{\mu_n}=\Phi^{-1}\operatorname{Co}(s_{p,\mu_n})\Phi,
\]
where $s_{p,\mu_n}(x):=s_p(x+\mu_n)$ for $x\in\mR$. Hence we must show that
\begin{equation}\label{eq:LO-modulations-1}
\lim_{n\to\infty}\|\operatorname{Co}(s_{p,\mu_n}-1)\psi\|_{L^p(\mR_+,d\mu)}=0
\end{equation}
for $\psi\in L^p(\mR_+,d\mu)$.

According to \cite[Chap.~III, Section~2.2]{S70}, for every $f\in L^p(\mR)$
and every $\varphi\in L^1(\mR)$ with $\int_\mR\varphi(x)dx=1$, we have
\begin{equation}\label{eq:LO-modulations-2}
\lim_{\eps\to 0}\|f*\varphi_\eps-f\|_{L^p(\mR)}=0,
\end{equation}
where $\varphi_\eps(x):=\eps^{-1}\varphi(x/\eps)$ for $x\in\mR$ and $\eps>0$.
Choosing now rapidly decreasing functions $\varphi$ in the Schwarz space
$\mathcal{S}(\mR)$ whose Fourier transforms $F\varphi$ have compact supports in
$\mR$, we derive from \eqref{eq:LO-modulations-2} that the set $Y$
of the functions $f\in L^2(\mR)\cap L^p(\mR)$, for which
$Ff$ has compact support in
$\mR$, is dense in $L^p(\mR)$. Hence the set $D$ of all functions
$\psi\in L^2(\mR_+,d\mu)\cap L^p(\mR_+,d\mu)$, for which the Mellin transform
$M\psi$ has compact support in $\mR$, is dense in $L^p(\mR_+,d\mu)$.
Obviously, it is sufficient to prove \eqref{eq:LO-modulations-1} for all
$\psi\in D$.

Fix $\psi\in D$. Since the support of $M\psi$ is compact, there exists a function
$\chi\in C_0^\infty(\mR)$ with a compact support $K$ such that
\begin{equation}\label{eq:LO-modulations-3}
\operatorname{Co}(s_{p,\mu_n}-1)\psi=M^{-1}\chi(s_{p,\mu_n}-1)M\psi.
\end{equation}
From Theorem~\ref{th:Stechkin} and \eqref{eq:def-E} it follows that
\begin{equation}\label{eq:LO-modulations-4}
\|M^{-1}\chi(s_{p,\mu_n}-1)M\psi\|_{L^p(\mR_+,d\mu)}
\le
c_p
\|\psi\|_{L^p(\mR_+,d\mu)}
\|\chi(s_{p,\mu_n}-1)\|_V,
\end{equation}
where $c_p:=\|S\|_{\cB(L^p(\mR))}$ and $\|\cdot\|_V:=\|\cdot\|_{L^\infty(\mR)}+V(\cdot)$.
It remains to show that
\begin{equation}\label{eq:LO-modulations-5}
\|\chi(s_{p,\mu_n}-1)\|_V
=
\|\chi(s_{p,\mu_n}-1)\|_{L^\infty(\mR)}+V\big(\chi(s_{p,\mu_n}-1)\big)
\to 0
\end{equation}
as $n\to\infty$. We have
\begin{equation}\label{eq:LO-modulations-6}
\|\chi(s_{p,\mu_n}-1)\|_{L^\infty(\mR)}
\le
\|\chi\|_{L^\infty(\mR)}\sup_{x\in K}|s_p(x+\mu_n)-1|
\end{equation}
and
\begin{align}
V\big(&\chi(s_{p,\mu_n}-1)\big)
=
\int_\mR\left|\frac{d}{dx}\big[\chi(x)(s_p(x+\mu_n)-1)\big]\right|dx
\nonumber
\\
&\le
\int_\mR|\chi(x)|\,|s_p'(x+\mu_n)|dx
+
\int_\mR|\chi'(x)|\,|s_p(x+\mu_n)-1|dx
\nonumber
\\
&\le
C\sup_{x\in K}|s_p'(x+\mu_n)|+
C'\sup_{x\in K}|s_p(x+\mu_n)-1|,
\label{eq:LO-modulations-7}
\end{align}
where
\[
C:=\int_\mR|\chi(x)|dx<\infty,\quad
C':=\int_\mR|\chi'(x)|dx<\infty.
\]
Taking into account that $\mu_n\to+\infty$ and
\[
s_p(x+\mu_n)=\coth[\pi(x+\mu_n+i/p)],
\
s_p'(x+\mu_n)=-\frac{\pi}{\sinh^2[\pi(x+\mu_n+i/p)]},
\]
we see that
\begin{equation}\label{eq:LO-modulations-8}
\lim_{n\to\infty}\sup_{x\in K}|s_p(x+\mu_n)-1|=0,
\quad
\lim_{n\to\infty}\sup_{x\in K}|s_p'(x+\mu_n)|=0.
\end{equation}
Combining \eqref{eq:LO-modulations-6}--\eqref{eq:LO-modulations-8},
we arrive at \eqref{eq:LO-modulations-5}. From
\eqref{eq:LO-modulations-3}--\eqref{eq:LO-modulations-5} we obtain
\eqref{eq:LO-modulations-1}. Since
$(E_{\mu_n}^{-1}SE_{\mu_n})^*=E_{\mu_n}^{-1}SE_{\mu_n}$, this completes
the proof for $s=+\infty$.

The proof for $s=-\infty$ is analogous.
\QED
\end{proof}
\section{Localization}\label{sec:localization}
\subsection{Algebras $\cZ$, $\cF$, and $\Lambda$}
Let us consider
\[
\begin{split}
\cZ &:=\alg_\cB\big\{I,S,cR,K:\ c\in SO(\mR_+),\ K\in\cK\big\},
\\
\cF &:=\alg_\cB\big\{aI,S,W_\alpha,W_\alpha^{-1}:a\in SO(\mR_+)\big\},
\\
\Lambda &:=\big\{A\in\cB:\ AC-CA\in\cK\text{ for all }C\in\cZ\big\}.
\end{split}
\]

It is easy to see that $\Lambda$ is a closed unital subalgebra of $\cB$.
\begin{theorem}[{\cite[Theorem~6.8]{KKLsufficiency}}]
\label{th:embeddings}
We have $\cK\subset\cZ\subset\cF\subset\Lambda$.
\end{theorem}
\begin{lemma}\label{le:alg-Lambda}
An operator $A\in\Lambda$ is Fredholm if and only if the coset $A^\pi:=A+\cK$
is invertible in the quotient algebra $\Lambda^\pi:=\Lambda/\cK$.
\end{lemma}
The proof is straightforward.
\subsection{Fredholmness of operators in the algebra $\Lambda$}
By \cite[Theorem~6.11]{KKLsufficiency}, the maximal ideal space $M(\cZ^\pi)$
of the commutative Banach algebra $\cZ^\pi:=\cZ/\cK$ is homeomorphic to the
set $\{-\infty,+\infty\}\cup(\Delta\times\mR)$.
Let
\begin{align*}
\cI_{\pm\infty}^\pi&:=\operatorname{id}_{\cZ^\pi}
\big\{P_\mp^\pi,(gR)^\pi\ : \ g\in SO(\mR_+)\big\},
\\
\cI_{\xi,x}^\pi&:=\big\{Z^\pi\in\cZ^\pi:\
(Z^\pi)\widehat{\hspace{2mm}}(\xi,x)=0\big\}\quad\text{for }\;(\xi,x)\in\Delta\times\mR,
\end{align*}
where $(Z^\pi)\widehat{\hspace{2mm}}$ is the Gelfand transform of $Z^\pi$, which was
explicitly given in \cite[Section~6]{KKLsufficiency}.
Further,
let $\cJ_{\pm\infty}^\pi$ and $\cJ_{\xi,x}^\pi$ be the
closed two-sided ideals of the Banach algebra $\Lambda^\pi$
generated by the ideals $\cI^\pi_{\pm\infty}$ and
$\cI^\pi_{\xi,x}$ of the algebra $\cZ^\pi$, respectively,
and put
\[
\Lambda^\pi_{\pm\infty}:=\Lambda^\pi/\cJ^\pi_{\pm\infty},
\quad
\Lambda^\pi_{\xi,x}:=\Lambda^\pi/\cJ^\pi_{\xi,x}
\]
for the corresponding quotient algebras.
\begin{theorem}[{\cite[Theorem~6.12]{KKLsufficiency}}]
\label{th:localization-realization}
An operator $A\in\Lambda$ is Fredholm on the space $L^p(\mR_+)$ if and only
if the following two conditions are fulfilled:
\begin{itemize}
\item[{\rm (i)}]
the cosets $A^\pi+\cJ_{\pm\infty}^\pi$ are invertible
in the quotient algebras ${\Lambda}_{\pm\infty}^\pi$,
respectively;

\item[{\rm(ii)}]
for every $(\xi,x)\in\Delta\times\mR$, the coset
$A^\pi+\cJ_{\xi,x}^\pi$ is invertible in the quotient
algebra ${\Lambda}_{\xi,x}^\pi$.
\end{itemize}
\end{theorem}
\subsection{Quotient algebras $\Lambda_{+\infty}$ and $\Lambda_{-\infty}$}
Let $\cJ_{\pm\infty}$ be the closed two-sided ideal of the algebra $\Lambda$
generated by the operator $P_{\mp}$ and the ideal $\cK$. By $\Lambda_{\pm\infty}$
denote the quotient algebra $\Lambda/\cJ_{\pm\infty}$.

It is not difficult to see that
$R\in\operatorname{id}_\cA\{P_-\}\cap\operatorname{id}_\cA\{P_+\}$.
Hence the ideals $\cJ_{\pm\infty}^\pi$ of the quotient algebra
$\Lambda^\pi$ can be also represented in the form
$\cJ_{\pm\infty}^\pi=\operatorname{id}_{\Lambda^\pi}\{P_\mp^\pi\}$,
and therefore, respectively,
\begin{equation}\label{id-form}
\cJ_{\pm\infty}^\pi=\{A^\pi:\ A\in\cJ_{\pm\infty}\}.
\end{equation}
\begin{lemma}\label{le:lifting}
Suppose $C_-,C_+\in\Lambda$.
\begin{enumerate}
\item[{\rm(a)}]
The invertibility of the coset $(C_+P_++C_-P_-)^\pi+\cJ_{+\infty}^\pi$ in the quotient
algebra $\Lambda_{+\infty}^\pi$ is equivalent to the invertibility
of the coset $C_++\cJ_{+\infty}$ in the quotient algebra $\Lambda_{+\infty}$.

\item[{\rm(b)}]
The invertibility of the coset $(C_+P_++C_-P_-)^\pi+\cJ_{-\infty}^\pi$ in the quotient
algebra $\Lambda_{-\infty}^\pi$ is equivalent to the invertibility
of the coset $C_-+\cJ_{-\infty}$ in the quotient algebra $\Lambda_{-\infty}$.
\end{enumerate}
\end{lemma}
\begin{proof}
(a) Consider the mapping $\varphi:\Lambda/\cJ_{+\infty}\to\Lambda^\pi/\cJ_{+\infty}^\pi$
given by
\[
A+\cJ_{+\infty}\mapsto A^\pi+\cJ_{+\infty}^\pi\quad(A\in\Lambda).
\]
Obviously, $\varphi$ is a homomorphism of $\Lambda/\cJ_{+\infty}$ onto
$\Lambda^\pi/\cJ_{+\infty}^\pi$. If $A^\pi\in\cJ_{+\infty}^\pi$, then
from \eqref{id-form} it follows that $A\in\cJ_{+\infty}$, and therefore
$\varphi$ is injective. So, $\varphi$ is an isomorphism. Then
$C_++\cJ_{+\infty}=\varphi^{-1}(C_+^\pi+\cJ_{+\infty}^\pi)$ is invertible in
$\Lambda/\cJ_{+\infty}$ if and only if $C_+^\pi+\cJ_{+\infty}^\pi$ is
invertible in $\Lambda^\pi/\cJ_{+\infty}^\pi$.
By the definition of the ideal $\cJ_{+\infty}^\pi$, we have
$(C_\pm P_-)^\pi\in\cJ_{+\infty}^\pi$. From this observation and
$P_++P_-=I$ it follows that
\[
(C_+P_++C_-P_-)^\pi+\cJ_{+\infty}^\pi=C_+^\pi+\cJ_{+\infty}^\pi,
\]
which finishes the proof of part (a). The proof of part (b) is analogous.
\QED
\end{proof}
\section{Invertibility of binomial functional operators}\label{sec:FO}
\subsection{Invertibility of functional operators with slowly oscillating data}
For $s\in\{0,\infty\}$, $a,b\in SO(\mR_+)$, and $\alpha\in SOS(\mR_+)$, put
\begin{align*}
L_*(s;a,b,\alpha)
&:=
\liminf\limits_{t\to s}
\left( |a(t)|-|b(t)|\big(\alpha'(t)\big)^{-1/p}\right),
\\
L^*(s;a,b,\alpha)
&:=
\limsup\limits_{t\to s}
\left(|a(t)|-|b(t)|\big(\alpha'(t)\big)^{-1/p}\right).
\end{align*}
\begin{theorem}[{\cite[Theorem~1.1]{KKLsufficiency}}]
\label{th:FO}
Suppose $a,b\in SO(\mR_+)$ and $\alpha\in SOS(\mR_+)$. The
functional operator $aI-bW_\alpha$ is invertible on the
Lebesgue space $L^p(\mR_+)$ if and only if
either
\begin{equation}\label{eq:FO-1}
\inf\limits_{t\in\mR_+}|a(t)|>0,
\quad
L_*(0;a,b,\alpha)>0,
\quad
L_*(\infty;a,b,\alpha)>0;
\end{equation}
or
\begin{equation}\label{eq:FO-2}
\inf\limits_{t\in\mR_+}|b(t)|>0,
\quad
L^*(0;a,b,\alpha)<0,
\quad
L^*(\infty;a,b,\alpha)<0.
\end{equation}
If \eqref{eq:FO-1} holds, then
\[
(aI-bW_\alpha)^{-1}=\sum_{n=0}^\infty (a^{-1}bW_\alpha)^n a^{-1}I.
\]
If \eqref{eq:FO-2} holds, then
\[
(aI-bW_\alpha)^{-1}=-W_\alpha^{-1}\sum_{n=0}^\infty (b^{-1}aW_\alpha^{-1})^n b^{-1}I.
\]
\end{theorem}
\subsection{Invertibility of auxiliary binomial functional operators}
\label{subsec:FO-aux}
Suppose $\alpha_0(t):=t$ and $\alpha_n(t):=\alpha[\alpha_{n-1}(t)]$ for $n\in\mZ$
and $t\in\mR_+$. Fix a point $\tau\in\mR_+$ and put
\[
\tau_-:=\lim_{n\to-\infty}\alpha_n(\tau),
\quad
\tau_+:=\lim_{n\to+\infty}\alpha_n(\tau).
\]
Then either $\tau_-=0$ and $\tau_+=\infty$, or $\tau_-=\infty$ and
$\tau_+=0$. Let $\gamma$ be a segment of $\mR_+$ with endpoints
$\tau$ and $\alpha(\tau)$.
Suppose $\chi_\gamma$ is the characteristic function of $\gamma$ and
$\widetilde{\chi}_\gamma$ is an arbitrary function in $C(\mR_+)$ with
nonempty support in $\gamma$. Consider the half-open intervals
\[
\gamma_-:=\bigcup_{k=1}^\infty\alpha_{-k}(\gamma),
\quad
\gamma_+:=\bigcup_{k=1}^\infty\alpha_k(\gamma).
\]
Let $\widetilde{\tau}_\pm$ denote the
endpoint of the half-open interval
$\gamma\cup\gamma_\pm$ such that $\widetilde{\tau}_\pm\ne\tau_\pm$, respectively.
Consider functions $\chi_\pm\in C(\overline{\mR}_+)$ such
that $\chi_-(t)=1$ for all $t\in\gamma_-$,
$\chi_+(t)=1$ for all $t\in\gamma_+$, and
$\chi_-(t)+\chi_+(t)=1$ for all $t\in\mR_+$.
\begin{lemma}\label{le:FO-aux}
Let $A=aI-bW_\alpha$ where $a,b\in SO(\mR_+)$ and $\alpha\in SOS(\mR_+)$.
\begin{enumerate}
\item[{\rm(a)}]
Suppose
\begin{equation}\label{eq:FO-aux-1}
L_*(\tau_-;a,b,\alpha)>0>L^*(\tau_+;a,b,\alpha),
\end{equation}
\begin{equation}\label{eq:FO-aux-2}
\inf_{t\in\gamma\cup\gamma_-}|a(t)|>0,
\quad
\inf_{t\in\gamma\cup\gamma_+}|b(t)|>0,
\end{equation}
and put
\begin{equation}\label{eq:FO-aux-3}
\widetilde{a}(t):=
\left\{\begin{array}{ll}
a(t) & \text{for }  t\in\gamma\cup\gamma_-,
\\
a(\widetilde{\tau}_-) & \text{otherwise},
\end{array}\right.
\
\widetilde{b}(t):=\left\{\begin{array}{lll}
b(t) & \text{for } t\in\gamma\cup\gamma_+,
\\
b(\widetilde{\tau}_+) & \text{otherwise}.
\end{array}\right.
\end{equation}
Then the operators
\begin{equation}\label{eq:FO-aux-4}
A_{1,\chi_-}:=\widetilde{a}I-b\chi_-W_\alpha,
\quad
A_{2,\chi_+}:=a\chi_+I-\widetilde{b}W_\alpha
\end{equation}
are invertible on the space $L^p(\mR_+)$ and
\begin{equation}\label{eq:FO-aux-5}
A\Pi_r=0,\quad
\widetilde{\chi}_\gamma\Pi_r=\widetilde{\chi}_\gamma I
\quad{for}\quad
\Pi_r:=(A_{1,\chi_-}^{-1}-A_{2,\chi_+}^{-1})a\chi_\gamma I.
\end{equation}

\item[{\rm(b)}]
Suppose
\begin{equation}\label{eq:FO-aux-6}
L^*(\tau_-;a,b,\alpha)<0<L_*(\tau_+;a,b,\alpha),
\end{equation}
\[
\inf_{t\in\gamma\cup\gamma_+}|a(t)|>0,
\quad
\inf_{t\in\gamma\cup\gamma_-}|b(t)|>0,
\]
and put
\begin{equation*}
\widetilde{a}(t):=
\left\{\begin{array}{ll}
a(t) & \text{for }  t\in\gamma\cup\gamma_+,
\\
a(\widetilde{\tau}_+) & \text{otherwise},
\end{array}\right.
\
\widetilde{b}(t):=\left\{\begin{array}{lll}
b(t) & \text{for } t\in\gamma\cup\gamma_-,
\\
b(\widetilde{\tau}_-) & \text{otherwise}.
\end{array}\right.
\end{equation*}
Then the operators
the operators
\[
A_{1,\chi_+\circ\alpha}:=\widetilde{a}I-b(\chi_+\circ\alpha)W_\alpha,
\quad
A_{2,\chi_-}:=a\chi_-I-\widetilde{b}W_\alpha,
\]
are invertible on the space $L^p(\mR_+)$ and
\[
\Pi_l A=0,
\quad
\Pi_l\widetilde{\chi}_\gamma I=\widetilde{\chi}_\gamma I
\quad{for}\quad
\Pi_l:=\chi_\gamma a(A_{1,\chi_+\circ\alpha}^{-1}-A_{2,\chi_-}^{-1}).
\]
\end{enumerate}
\end{lemma}
\begin{proof}
(a) The idea of the proof is borrowed from \cite[Lemma~3]{K84}. Clearly,
the functions defined by \eqref{eq:FO-aux-3} belong to $SO(\mR_+)$.
From \eqref{eq:FO-aux-1}--\eqref{eq:FO-aux-3} it follows that
\[
\inf_{t\in\mR_+}|\widetilde{a}(t)|>0,
\quad
\inf_{t\in\mR_+}|\widetilde{b}(t)|>0,
\quad
L_*(\tau_\pm;\widetilde{a},b\chi_-,\alpha)>0,
\quad
L^*(\tau_\pm;a\chi_+,\widetilde{b},\alpha)<0.
\]
By Theorem~\ref{th:FO}, the operators \eqref{eq:FO-aux-4}
are invertible on the space $L^p(\mR_+)$, and
\begin{equation}\label{eq:FO-aux-8}
A_{1,\chi_-}^{-1}=\sum_{n=0}^\infty (\widetilde{a}^{-1}b\chi_-W_\alpha)^n
\widetilde{a}^{-1}I,
\
A_{2,\chi_+}^{-1}=-W_\alpha^{-1}\sum_{n=0}^\infty (\widetilde{b}^{-1}a\chi_+
W_\alpha^{-1})^n \widetilde{b}^{-1}I.
\end{equation}
Further, in view of \eqref{eq:FO-aux-3}, we get the relations
\begin{equation}\label{eq:FO-aux-9}
\begin{aligned}
&
(aI-bW_\alpha)W_\alpha^n\chi_\gamma I=(\widetilde{a}I-b\chi_-W_\alpha)
W_\alpha^n\chi_\gamma I & (n\in\mN\cup\{0\}),
\\
&
(aI-bW_\alpha)(W_\alpha^{-1})^n\chi_\gamma I=(a\chi_+I-\widetilde{b}W_\alpha)
(W_\alpha^{-1})^n \chi_\gamma I & (n\in\mN).
\end{aligned}
\end{equation}
Applying \eqref{eq:FO-aux-8} and \eqref{eq:FO-aux-9} we infer that
\begin{align*}
AA_{1,\chi_-}^{-1}a\chi_\gamma I
&=A_{1,\chi_-}A_{1,\chi_-}^{-1}a\chi_\gamma I=a\chi_\gamma I,
\\
AA_{2,\chi_+}^{-1}a\chi_\gamma I
&=A_{2,\chi_+}A_{2,\chi_+}^{-1}a\chi_\gamma I=a\chi_\gamma I,
\end{align*}
whence $A\Pi_r=A(A_{1,\chi_-}^{-1}-A_{2,\chi_+}^{-1})a\chi_\gamma I=0$. On
the other hand, since
\[
\widetilde{\chi}_\gamma W_\alpha^n\chi_\gamma I=0,\quad
\widetilde{\chi}_\gamma(W_\alpha^{-1})^n\chi_\gamma I=0\quad(n\in\mN),
\]
we deduce from \eqref{eq:FO-aux-3} that
\begin{align*}
\widetilde{\chi}_\gamma\Pi_r
&=
\widetilde{\chi}_\gamma\sum_{n=0}^\infty
(\widetilde{a}^{-1}b\chi_-W_\alpha)^n\widetilde{a}^{-1}a\chi_\gamma I +\widetilde{\chi}_\gamma W_\alpha^{-1}
\sum_{n=0}^\infty(\widetilde{b}^{-1}a\chi_+W_\alpha^{-1})^n \widetilde{b}^{-1}a\chi_\gamma I
\\
&=
\widetilde{\chi}_\gamma\chi_\gamma I=\widetilde{\chi}_\gamma I,
\end{align*}
which completes the proof of \eqref{eq:FO-aux-5}. Part (a) is proved.

The proof of part (b) is similar and therefore is omitted.
\QED
\end{proof}
\subsection{Invertibility of functional operators with multiplicative shifts}
\begin{lemma}\label{le:binom-mult}
Let $a,b\in SO(\mR_+)$ and $\alpha:\mR_+\to\mR_+$ be a multiplicative shift
given by $\alpha(t)=kt$ for all $t\in\mR_+$ with some $k\in\mR_+$. The
following statements are equivalent:
\begin{enumerate}
\item[{\rm(i)}]
the functional operator $aI-bW_\alpha$ is invertible on the space $L^p(\mR_+)$;
\item[{\rm(ii)}]
the coset $aI-bW_\alpha+\cJ_{+\infty}$ is invertible in the quotient algebra $\Lambda_{+\infty}$;
\item[{\rm(iii)}]
the coset $aI-bW_\alpha+\cJ_{-\infty}$ is invertible in the quotient algebra $\Lambda_{-\infty}$.
\end{enumerate}
\end{lemma}
\begin{proof}
(i) $\Rightarrow$ (ii). Clearly, if $A:=aI-bW_\alpha$ is invertible on
$L^p(\mR_+)$, then it is Fredholm. Moreover, $A\in\Lambda$ by
Theorem~\ref{th:embeddings}. Then from Lemma~\ref{le:alg-Lambda}
we see that there exists $B\in\Lambda$ such that $AB-I\in\cK$ and
$BA-I\in\cK$. Since the ideal $\cJ_{+\infty}$ contains $\cK$,
the latter relations imply that the coset $A+\cJ_{+\infty}$ is invertible in
the quotient algebra $\Lambda_{+\infty}$ and $B+\cJ_{+\infty}$ is its inverse.

(ii) $\Rightarrow$ (i). Consider the test sequence $\nu:=\{\nu_n\}_{n=1}^\infty\subset\mR$
relative to $+\infty$ and given by
\[
\nu_n:=\left\{\begin{array}{lll}
2\pi n|\log k|^{-1} &\mbox{if} &k\ne 1, \\
2\pi n &\mbox{if} & k=1
\end{array}\right.
\quad(n\in\mN).
\]
Then for every $n\in\mN$, $f\in L^p(\mR_+)$, and $t\in\mR_+$,
\[
(E_{\nu_n}^{-1}W_\alpha E_{\nu_n}f)(t)
=
t^{-i\nu_n}(kt)^{i\nu_n}f(kt)=k^{i\nu_n}f(kt)=(W_\alpha f)(t)
\]
because $k^{i\nu_n}=1$. That is, $E_{\nu_n}^{-1}W_\alpha E_{\nu_n}=W_\alpha$.
Thus, the limit operator $(W_\alpha)_{\cE_\nu^{+\infty}}$ with respect to the
sequence of pseudoisometries $\cE_\nu^{+\infty}:=\{E_{\nu_n}\}_{n=1}^\infty$
exists and
\begin{equation}\label{eq:binom-mult-1}
(W_\alpha)_{\cE_\nu^{+\infty}}=W_\alpha.
\end{equation}
From Lemma~\ref{le:LO-modulations} we obtain
\begin{equation}\label{eq:binom-mult-2}
(aI)_{\cE_\nu^{+\infty}}=aI,
\quad
(bI)_{\cE_\nu^{+\infty}}=bI,
\quad
S_{\cE_\nu^{+\infty}}=I.
\end{equation}
On the other hand, by Lemma~\ref{le:LO-compact-modulations},
\begin{equation}\label{eq:binom-mult-3}
K_{\cE_\nu^{+\infty}}=0
\quad\mbox{for every}\quad K\in\cK.
\end{equation}
Combining \eqref{eq:binom-mult-1}--\eqref{eq:binom-mult-3}
with Lemma~\ref{le:LO-properties}(b)--(c), we see that
\begin{equation}\label{eq:binom-mult-4}
(aI-bW_\alpha)_{\cE_\nu^{+\infty}}=aI-bW_\alpha
\end{equation}
and $J_{\cE_\nu^{+\infty}}=0$ for all $J\in\cJ_{+\infty}$ (recall that
the ideal $\cJ_{+\infty}$ is generated by $\cK$ and $P_-=(I-S)/2$,
so $(P_-)_{\cE_\nu^{+\infty}}=0$).

Since the coset $aI-bW_\alpha+\cJ_{+\infty}$ is invertible
in the quotient algebra $\Lambda_{+\infty}$, we deduce from
Theorem~\ref{th:inv-quotient-algebra} that the limit
operator \eqref{eq:binom-mult-4} is invertible. This finishes
the proof of the implication (ii) $\Rightarrow$ (i).
Thus, the equivalence (i) $\Leftrightarrow$ (ii) is proved.

The proof of the equivalence (i) $\Leftrightarrow$ (iii) is analogous.
\QED
\end{proof}
\section{Necessary conditions for Fredholmness}\label{sec:necessity}
\subsection{Necessity of condition (i)}
In this subsection we prove that condition (i) in Theorem~\ref{th:sufficiency}
is necessary for the Fredholmness of the operator $N$. We start
with the following auxiliary result.
\begin{lemma}\label{le:compactness}
Let the functions $\chi_\gamma$ and $\widetilde{\chi}_\gamma$ be as in
Section~{\rm\ref{subsec:FO-aux}}, and $n\in\mZ$.
\begin{enumerate}
\item[{\rm (a)}]
The operators $(\chi_\gamma\circ\alpha_n) P_- P_+$ and
$P_- P_+(\chi_\gamma\circ\alpha_n)I$ are compact.

\item[{\rm(b)}]
The operators $\widetilde{\chi}_\gamma P_\pm$ and $P_\pm\widetilde{\chi}_\gamma I$
are not compact.
\end{enumerate}
\end{lemma}
\begin{proof}
(a) In view of \eqref{eq:S-R-relation}, we have
$P_- P_+\in\operatorname{id}_\cA\{R\}$. Let $c$ be
a continuous function equal $1$ on the support of
$\chi_\gamma\circ\alpha_n$ and vanishing at $0$ and $\infty$.
Then the operator $(c\circ\alpha_n)P_- P_+$ is compact
by \cite[Corollary~6.6]{KKLsufficiency}. Therefore the operator
$(\chi_\gamma\circ\alpha_n)P_- P_+=(\chi_\gamma\circ\alpha_n)(c\circ\alpha_n)P_- P_+$
is also compact. The compactness of
$P_- P_+(\chi_\gamma\circ\alpha_n)I$ is proved analogously
with the aid of Theorem~\ref{th:compactness-commutators}.
Part (a) is proved.
Part (b) follows from \cite[Theorem~4.1(c)]{RS90}.
\QED
\end{proof}
\begin{lemma}\label{le:nec-1}
Suppose $a,b,c,d\in SO(\mR_+)$, $\alpha\in SOS(\mR_+)$, and the operator $N$
is given by \eqref{eq:def-N}.
\begin{enumerate}
\item[{\rm (a)}]
If $N$ is Fredholm, then either
\begin{equation}\label{eq:nec-1-1}
L_*(0;a,b,\alpha)>0,
\quad
L_*(\infty;a,b,\alpha)>0;
\end{equation}
or
\begin{equation}\label{eq:nec-1-2}
L^*(0;a,b,\alpha)<0,
\quad
L^*(\infty;a,b,\alpha)<0.
\end{equation}

\item[{\rm (b)}]
If $N$ is Fredholm, then either
\begin{equation}\label{eq:nec-1-3}
L_*(0;c,d,\alpha)>0,
\quad
L_*(\infty;c,d,\alpha)>0;
\end{equation}
or
\begin{equation}\label{eq:nec-1-4}
L^*(0;c,d,\alpha)<0,
\quad
L^*(\infty;c,d,\alpha)<0.
\end{equation}
\end{enumerate}
\end{lemma}
\begin{proof}
(a)
Fix $s\in\{0,\infty\}$ and $\xi\in M_s(SO(\mR_+))$. By
Lemma~\ref{le:LO-dilations}, there exists a test sequence
$h^\xi=\{h_n^\xi\}_{n=1}^\infty\subset\mR_+$ relative to the point
$s$ such that the limit operator $N_{\cV_{h^\xi}^s}$ with respect
to the sequence of pseudoisometries
$\cV_{h^\xi}^s:=\{V_{h_n^\xi}\}_{n=1}^\infty\subset\cB$ exists and
\begin{equation}\label{eq:def-limit-N}
N_{\cV_{h^\xi}^s}= \big(a(\xi)I-b(\xi)W_{\alpha_\xi}\big)P_+ +
\big(c(\xi)I-d(\xi)W_{\alpha_\xi}\big)P_-,
\end{equation}
where $\alpha_\xi(t)=e^{\omega(\xi)}t$ is a multiplicative shift
and $\omega(t)=\log[\alpha(t)/t]$ belongs to $SO(\mR_+)$ in view
of Lemma~\ref{le:exp-repr}. From
Lemma~\ref{le:LO-compact-dilations} it follows that
$K_{\cV_{h^\xi}^s}=0$ for every $K\in\cK$. Since the operator $N$
is Fredholm, the coset $N^\pi=N+\cK$ is invertible in the quotient
algebra $\Lambda^\pi$ in view of Lemma~\ref{le:alg-Lambda}.
Applying Theorem~\ref{th:inv-quotient-algebra} with
$\fA=\Lambda$, $\fJ=\cK$, $A=N$, and $\cU=\cV_{h^\xi}^s$, we
conclude that the operator $N_{\cV_ {h^\xi}^s}$ is invertible.

Obviously, $N_{\cV_{h^\xi}^s}$ is Fredholm. Then from
Theorem~\ref{th:localization-realization} it follows that the coset
$(N_{\cV_{h^\xi}^s})^\pi+\cJ_{+\infty}^\pi$ is invertible in the quotient
algebra $\Lambda_{+\infty}^\pi$. By Lemma~\ref{le:lifting}, this is equivalent
to the invertibility of the coset $a(\xi)I-b(\xi)W_{\alpha_\xi}+\cJ_{+\infty}$
in the quotient algebra $\Lambda_{+\infty}$. Since $\alpha_\xi(t)=e^{\omega(\xi)}t$
is a multiplicative shift, from Lemma~\ref{le:binom-mult} it follows that
the above condition is equivalent to the invertibility of the operator
$a(\xi)I-b(\xi)W_{\alpha_\xi}$. Applying Theorem~\ref{th:FO} to this operator
and taking into account Lemma~\ref{le:SOS-derivative},
we obtain either
\begin{equation}\label{eq:nec-1-5}
|a(\xi)|-|b(\xi)|\big(\alpha'(\xi)\big)^{-1/p}=
|a(\xi)|-|b(\xi)|\big(e^{\omega(\xi)}\big)^{-1/p}>0
\end{equation}
or
\begin{equation}\label{eq:nec-1-6}
|a(\xi)|-|b(\xi)|\big(\alpha'(\xi)\big)^{-1/p}=
|a(\xi)|-|b(\xi)|\big(e^{\omega(\xi)}\big)^{-1/p}<0.
\end{equation}
The fibers $M_s(SO(\mR_+))$ are
connected compact Hausdorff spaces by Lemma~\ref{le:connected-fibers}.
Since $a(\xi),b(\xi)$,
and $\alpha'(\xi)$ depend continuously on $\xi\in M_s(SO(\mR_+))$, we deduce
that if $N$ is Fredholm, then for every $s\in\{0,\infty\}$
either \eqref{eq:nec-1-5} holds for all $\xi\in M_s(SO(\mR_+))$
or \eqref{eq:nec-1-6} holds for all $\xi\in M_s(SO(\mR_+))$.
Hence we conclude from Lemma~\ref{le:SO-fundamental-property}
that for each $s\in\{0,\infty\}$ either $L_*(s;a,b,\alpha)>0$
or $L^*(s;a,b,\alpha)<0$.

It remains to prove that actually either \eqref{eq:nec-1-1}
or \eqref{eq:nec-1-2} is fulfilled, that is, to show that
$L^*(0;a,b,\alpha)<0<L_*(\infty;a,b,\alpha)$ or
$L^*(\infty;a,b,\alpha)<0<L_*(0;a,b,\alpha)$
are impossible. Since either $\tau_-=0$ and $\tau_+=\infty$,
or $\tau_-=\infty$ and $\tau_+=0$, the latter inequalities
take either the form \eqref{eq:FO-aux-1}, or the form \eqref{eq:FO-aux-6}.

On the contrary, suppose \eqref{eq:FO-aux-1} is fulfilled. Then
there are open neighborhoods $u(\tau_\pm)\subset\mR_+$ of
$\tau_\pm$ such that $|a|$ is separated from zero on $u(\tau_-)$
and $|b|$ is separated from zero on $u(\tau_+)$. Take a segment $\gamma\subset
\mR_+\setminus\overline{u(\tau_-)\cup u(\tau_+)}$ with endpoints
$\tau$ and $\alpha(\tau)$. Since $N$ is Fredholm, by a small
perturbation of coefficients $a,b$ in $C\big(\mR_+
\setminus(u(\tau_-)\cup u(\tau_+))\big)$ we can achieve
the fulfillment of \eqref{eq:FO-aux-2} for perturbed coefficients
keeping the operator $N$ Fredholm. Notice that inequalities \eqref{eq:FO-aux-1}
remain valid for perturbed coefficients. Let us save notation
$a,b$ for perturbed coefficients. Then in virtue of Lemma~\ref{le:FO-aux}(a)
we obtain the operator $\Pi_r$ given by \eqref{eq:FO-aux-5}. Setting now
$A_+:=aI-bW_\alpha$ and $A_-:=cI-dW_\alpha$ and taking into account
Theorem~\ref{th:compactness-commutators}, we get
\begin{equation}\label{eq:nec-1-7}
NP_+=(A_+P_++A_-P_-)P_+\simeq P_+A_++(A_--A_+)P_-P_+,
\end{equation}
where $C\simeq D$ means that $C-D$ is a compact operator.
Recall that since $N$ is Fredholm, there is an operator $N^{(-1)}\in\cB$,
called a regularizer of $N$, such that $NN^{(-1)}\simeq N^{(-1)}N\simeq I$.
Then, applying $\Pi_r$ and $N^{(-1)}$, we infer from \eqref{eq:nec-1-7} that
\begin{equation}\label{eq:nec-1-8}
P_+\Pi_r\simeq N^{(-1)}NP_+\Pi_r\simeq N^{(-1)}P_+A_+\Pi_r+N^{(-1)}
(A_--A_+)P_-P_+\Pi_r.
\end{equation}
From Lemma~\ref{le:compactness}(a) and
\[
\Pi_r=\sum_{n=0}^\infty(\chi_\gamma\circ\alpha_n)
(a^{-1}b\chi_-W_\alpha)^n +(\chi_\gamma\circ\alpha_{-n-1})W_\alpha^{-1}
\sum_{n=0}^\infty(b^{-1}a\chi_+W_\alpha^{-1})^n b^{-1}aI
\]
we get $P_-P_+\Pi_r\simeq 0$.
On the other hand, by \eqref{eq:FO-aux-5} with $A=A_+$,
we obtain $A_+\Pi_r=0$. The latter two relations imply in view
of \eqref{eq:nec-1-8} that $P_+\Pi_r\simeq 0$. Let $\widetilde{\chi}_\gamma$
be as in Section~\ref{subsec:FO-aux}. From \eqref{eq:FO-aux-5} and
Theorem~\ref{th:compactness-commutators} we get
\[
P_+\widetilde{\chi}_\gamma I=P_+\widetilde{\chi}_\gamma\Pi_r
\simeq\widetilde{\chi}_\gamma P_+\Pi_r\simeq 0.
\]
Hence $P_+\widetilde{\chi}_\gamma I$ is a compact operator, which
is impossible due to Lemma~\ref{le:compactness}(b).

Analogously, if \eqref{eq:FO-aux-6} holds, then applying Lemma~\ref{le:FO-aux}(b) we
conclude that
\[
\Pi_l P_+\simeq \Pi_l P_+ N N^{(-1)}\simeq \Pi_l(A_+P_++P_-P_+(A_--A_+))
N^{(-1)}\simeq 0,
\]
and hence
\[
\widetilde{\chi}_\gamma P_+=\Pi_l\widetilde{\chi}_\gamma P_+
\simeq \Pi_l P_+\widetilde{\chi}_\gamma I\simeq 0,
\]
which again is impossible. Thus, either \eqref{eq:nec-1-1} or
\eqref{eq:nec-1-2} holds, and hence part  (a) is
proved. The proof of part (b) is analogous.
\QED
\end{proof}
\begin{lemma}\label{le:nec-2}
Suppose $a,b,c,d\in SO(\mR_+)$, $\alpha\in SOS(\mR_+)$, and the operator $N$
is given by \eqref{eq:def-N}.
\begin{enumerate}
\item[{\rm (a)}]
If $N$ is Fredholm and \eqref{eq:nec-1-1} is fulfilled, then
$\inf\limits_{t\in\mR_+}|a(t)|>0$.

\item[{\rm (b)}]
If $N$ is Fredholm and \eqref{eq:nec-1-2} is fulfilled, then
$\inf\limits_{t\in\mR_+}|b(t)|>0$.

\item[{\rm (c)}]
If $N$ is Fredholm and \eqref{eq:nec-1-3} is fulfilled, then
$\inf\limits_{t\in\mR_+}|c(t)|>0$.

\item[{\rm (d)}]
If $N$ is Fredholm and \eqref{eq:nec-1-4} is fulfilled, then
$\inf\limits_{t\in\mR_+}|d(t)|>0$.
\end{enumerate}
\end{lemma}
\begin{proof}
(a) Assume the contrary, that is, $\inf\limits_{t\in\mR_+}|a(t)|=0$. From
\eqref{eq:nec-1-1} it follows that there exist numbers $0<m<M<\infty$ such
that the function $a$ is bounded away from zero on $(0,m]\cup[M,\infty)$.
Hence there is a point $t_0\in(m,M)$ such that $a(t_0)=0$. Fix some
$\tau$ such that $t_0$ belongs to the interior of the segment $\gamma$
with the endpoints $\tau$ and $\alpha(\tau)$. Choose $m$ and $M$
such that $\gamma\subset[m,M]$. Suppose $u=u(t_0)$ is a
closed neighborhood of the point $t_0$ that is contained in $\gamma$ and whose
endpoints do not coincide with $\tau$ and $\alpha(\tau)$. Then there exists
a continuous function $\chi_u$ supported in $u$ and such
that $\chi_u(t_0)=1$.

Let $\varphi,\psi\in SO(\mR_+)$ be functions such that
\begin{enumerate}
\item
$a(t)=\varphi(t)=\psi(t)$ for $t\in(0,m]\cup[M,\infty)$;

\item
$\varphi(t)=0$ for $t\in u$ and $\varphi(t)\ne 0$ for $t\in\mR\setminus u$;

\item
$\psi(t)\ne 0$ for $t\in\mR$;

\item
$\varphi(t)=\psi(t)$ for $t\in\mR\setminus\gamma$.
\end{enumerate}
Consider the operator
\[
\widetilde{N}=(\varphi I-bW_\alpha)P_++(cI-dW_\alpha)P_-.
\]
It is clear that $\|\widetilde{N}-N\|_\cB=O(\|\varphi-a\|_{L^\infty(\mR_+)})$.
Since  $\varphi$ can be chosen arbitrarily close to $a$ in the norm of
$L^\infty(\mR_+)$ and $N$ is Fredholm, we can guarantee that $\widetilde{N}$
is also Fredholm for some $\varphi$ as above.

By Theorem~\ref{th:localization-realization}, the coset
$\widetilde{N}^\pi+\cJ_{+\infty}^\pi$ is invertible in the quotient algebra
$\Lambda_{+\infty}^\pi$. Therefore, the coset $\varphi I-bW_\alpha+\cJ_{+\infty}$
is invertible in the quotient algebra $\Lambda_{+\infty}$ in view of
Lemma~\ref{le:lifting}(a). Then there exists an operator $B\in\Lambda$ such
that
\begin{equation}\label{eq:nec-2-1}
(\varphi I-bW_\alpha+\cJ_{+\infty})(B+\cJ_{+\infty})=I+\cJ_{+\infty}.
\end{equation}
On the other hand, $\inf\limits_{t\in\mR_+}|\psi(t)|>0$, and by
\eqref{eq:nec-1-1} we have for $s\in\{0,\infty\}$,
\[
L_*(s;\psi,b,\alpha)=L_*(s;a,b,\alpha)>0.
\]
By Theorem~\ref{th:FO}, the operator $\psi I-bW_\alpha$ is invertible and
\[
(\psi I-bW_\alpha)^{-1}
=
\sum_{n=0}^\infty\left(\frac{b}{\psi}W_\alpha\right)^n\frac{1}{\psi}I
=
\frac{1}{\psi}\sum_{n=0}^\infty\left(\frac{b}{\psi\circ\alpha}W_\alpha\right)^n.
\]
Let
\[
C:=\chi_u\psi(\psi I-bW_\alpha)^{-1}.
\]
From Theorem~\ref{th:embeddings} we see that $C\in\cF\subset\Lambda$.
From the choice of $\varphi$ and $\psi$ it follows that $\chi_u\varphi=0$
and $\chi_u(\varphi\circ\alpha_k)=\chi_u(\psi\circ\alpha_k)$ for all $k\in\mN$.
Therefore,
\[
C(\varphi I-bW_\alpha)=
\chi_u\left(\sum_{n=0}^\infty\frac{b}{\psi\circ\alpha}W_\alpha\right)
(\varphi I-bW_\alpha)=\chi_u\varphi I=0.
\]
Hence
\begin{equation}\label{eq:nec-2-2}
(C+\cJ_{+\infty})(\varphi I-bW_\alpha+\cJ_{+\infty})=\cJ_{+\infty}.
\end{equation}
Multiplying \eqref{eq:nec-2-2} from the right by $B+\cJ_{+\infty}$ and taking
into account \eqref{eq:nec-2-1}, we obtain $C+\cJ_{+\infty}=\cJ_{+\infty}$.
Then $C\in\cJ_{+\infty}$.

It is clear that $\chi_u\circ\alpha_k=0$ for $k\in\mN$. Then
\[
C\chi_u I=\chi_u\sum_{n=0}^\infty\left(\frac{b}{\psi\circ\alpha}W_\alpha\right)^n\chi_u I
=\chi_u^2 I\in\cJ_{+\infty}.
\]
From Lemmas~\ref{le:LO-compact-modulations}, \ref{le:LO-modulations},
and~\ref{le:LO-properties} it follows that for an arbitrary sequence of
pseudoisometries $\cE_\mu^{+\infty}=\{E_{\mu_n}\}_{n=1}^\infty\subset\cB$,
the limit operators for all operators $J\in\cJ_{+\infty}$ are equal to zero.
In particular, then $(\chi_u^2 I)_{\cE_\mu^{+\infty}}=0$.
On the other hand, since $\chi_u^2\in SO(\mR_+)$, from Lemma~\ref{le:LO-modulations}
it also follows that $(\chi_u^2 I)_{\cE_\mu^{+\infty}}=\chi_u^2I\ne 0$.
This contradiction shows that $\inf\limits_{t\in\mR_+}|a(t)|>0$. Part (a)
is proved.

\medskip
(b) If the operator $N$ is Fredholm, then the operator
\begin{align}
-W_\alpha^{-1}[(aI &-bW_\alpha)P_++(cI-dW_\alpha)P_-]
\nonumber
\\
&=
[(b\circ\beta)I-(a\circ\beta)W_\beta]P_+
+[(d\circ\beta)I-(c\circ\beta)W_\beta]P_-
\label{eq:nec-2-3}
\end{align}
is also Fredholm. Recall that $\beta\in SOS(\mR_+)$ by Lemma~\ref{le:SOS-inverse}.
Then from Lemma~\ref{le:continuous-SOS} we see that $a\circ\beta,b\circ\beta\in SO(\mR_+)$.
Since $\beta$ preserves the orientation, has only
two fixed points $0$ and $\infty$, and $\log\alpha'$ is bounded,
we obtain for $s\in\{0,\infty\}$,
\[
\begin{split}
L^*(s;a,b,\alpha)
&=-\liminf_{t\to s}
\big((\alpha'\circ\beta)(t)\big)^{-1/p}
\Big(
|(b\circ\beta)(t)|
-
|(a\circ\beta)(t)|(\beta'(t))^{-1/p}
\Big)
\\
&\ge
-\sup_{t\in\mR_+}
\big((\alpha'\circ\beta)(t)\big)^{-1/p}
L_*(s;b\circ\beta,a\circ\beta,\beta).
\end{split}
\]
Hence \eqref{eq:nec-1-2} implies that $L_*(s;b\circ\beta,a\circ\beta,\beta)>0$ for $s\in\{0,\infty\}$.
Applying part (a) to the operator \eqref{eq:nec-2-3}, we obtain
\[
0<\inf_{t\in\mR_+}|(b\circ\beta)(t)|=\inf_{t\in\mR_+}|b(t)|.
\]
Part (b) is proved.
Parts (c) and (d) are proved by analogy with parts (a) and (b), respectively.
\QED
\end{proof}
Combining Lemmas~\ref{le:nec-1}--\ref{le:nec-2} and
Theorem~\ref{th:FO}, we arrive at the following part of
Theorem~\ref{th:main}.
\begin{theorem}\label{th:nec-i}
Suppose $a,b,c,d\in SO(\mR_+)$, $\alpha\in SOS(\mR_+)$, and the operator $N$
is given by \eqref{eq:def-N}.  If the operator $N$ is Fredholm on the space
$L^p(\mR_+)$, then the functional operators $A_+:=aI-bW_\alpha$ and
$A_-:=cI-dW_\alpha$ are invertible on the space $L^p(\mR_+)$.
\end{theorem}
\subsection{Necessity of condition (ii)}
To finish the proof of Theorem~\ref{th:main}, it remains to
prove the following.
\begin{theorem}
\label{th:nec-ii}
Suppose $a,b,c,d\in SO(\mR_+)$, $\alpha\in SOS(\mR_+)$, the operator $N$
is given by \eqref{eq:def-N}, and for every $\xi\in\Delta$ the function
$n_\xi$  is defined by \eqref{eq:def-n}.
If the operator $N$ is Fredholm on the space $L^p(\mR_+)$, then
$n_\xi(x)\ne 0$ for every pair $(\xi,x)\in\Delta\times\mR$.
\end{theorem}
\begin{proof}
Fix $\xi\in\Delta$. In the proof of Lemma~\ref{le:nec-1} it was shown that
if $N$ is Fredholm, then the operator $N_{\cV_{h^\xi}^s}$ given by
\eqref{eq:def-limit-N}
with $\alpha_\xi(t)=e^{\omega(\xi)}t$ is invertible.
On the other hand, taking into account
Theorem~\ref{th:algebra-A} and Lemma~\ref{le:mult-shift-convolution},
we see that $N_{\cV_{h^\xi}^s}=\Phi^{-1} \operatorname{Co}(n_\xi)\Phi$,
where $n_\xi\in SAP_p$ is given by \eqref{eq:def-n}. Hence,
$\operatorname{Co}(n_\xi)$ is invertible on the space $L^p(\mR_+,d\mu)$. Then
from Theorem~\ref{th:invertibility-convolution} we deduce that
$\inf\limits_{x\in\mR}|n_\xi(x)|>0$. Since $s\in\{0,\infty\}$ and $\xi
\in M_s(SO(\mR_+))$ were chosen arbitrarily, we conclude that
$n_\xi(x)\ne 0$ for all $(\xi,x)\in\Delta\times\mR$.
\QED
\end{proof}
\subsection*{Acknowledgment}
This work is partially supported by ``Centro de An\'alise Funcional e Aplica\c{c}\~oes"
at Instituto Superior T\'ecnico (Lisboa, Portugal), which is financed by FCT (Portugal).
The second author is also supported by the SEP-CONACYT Project No. 25564 (M\'exico)
and by PROMEP (M\'exico) via ``Proy\-ecto de Redes".

\end{document}